\documentclass[12pt]{iopart}

\expandafter\let\csname equation*\endcsname\relax
\expandafter\let\csname endequation*\endcsname\relax
  
\usepackage{amsmath}
\usepackage{amsfonts}
\usepackage{amssymb}
\usepackage{amsthm}
\usepackage{stmaryrd}
\usepackage{bm}
\usepackage{cite}
\usepackage[english]{babel}
\usepackage{dblaccnt}
\usepackage{accents}
\usepackage{graphicx}
\usepackage{subfig}
\usepackage{amscd}
\usepackage[mathscr]{euscript}
\usepackage{graphicx, color,float}
\usepackage[toc,page]{appendix}
\usepackage{tikz}
\usepackage{comment}

\newtheorem{defi}{Definition}[section]
\newtheorem{lemma}[defi]{Lemma}
\newtheorem{theorem}[defi]{Theorem}

\newtheorem{asp}{Assumption}  

\newtheorem{definition}[defi]{Definition}

\newcommand{\opef}{{\mathrm{F}}}
\renewcommand{\BB}{{\mathrm{B}}}

\newcommand{\Dint}{{\Lambda}}
\newcommand{\Dpint}{{\mathcal{O}}}
\newcommand{\Dpout}{{\mathcal{O}^c}}
\newcommand{\Dpoutp}{{\mathcal{O}^c_p}}

\newcommand{\Dtot}{{\widehat{D}}}

\newcommand{\ftm}{{f_1}}
\newcommand{\fte}{{f_2}}
\newcommand{\f}{{(\ftm, \fte)}}
\newcommand{\fugc}{{\widetilde{\fug}}}

\renewcommand{\i}{\mathrm{i}}
\renewcommand{\d}[1]{\,\mathrm{d}#1 \,}

\renewcommand{\Re}{\mathrm{Re}\,}
\renewcommand{\Im}{\mathrm{Im}\,}

\newcommand{\ol}[1]{\overline{#1}}

\renewcommand{\epsilon}{\varepsilon}
\newcommand{\loc}{\mathrm{loc}}

\newcommand*{\N}{\ensuremath{\mathbb{N}}}
\newcommand*{\Z}{\ensuremath{\mathbb{Z}}}
\newcommand*{\R}{\ensuremath{\mathbb{R}}}

\newcommand{\normM}[1]{{\llbracket {#1} \rrbracket}}

\newcommand{\interd}[2]{\ensuremath{\llbracket {#1}, {#2} \rrbracket}}

\newcommand{\I}{\mathcal{I}}

\newcommand{\CC}{\ensuremath{\mathbb{C}}}

\newcommand{\sper}{\#}

\newcommand{\dsp}{\displaystyle}

\newcommand{\opeh}{{\cal H}}
\newcommand{\opei}{{\mathrm{I}}}
\newcommand{\opeg}{{\mathrm{G}}}
\newcommand{\open}{{\mathrm{N}}}

\newcommand{\opet}{{\mathrm{T}}}

\newcommand{\Odual}[1]{(#1)^{*}}

\newcommand{\alpm}[1]{{\alpha_\sper(#1)}}
\newcommand{\betm}[1]{{\beta_\sper(#1)}}
\newcommand{\olbetm}[1]{{\ol \beta_\sper(#1)}}
\newcommand{\uip}{u^{i,+}}
\newcommand{\uim}{u^{i,-}}
\newcommand{\uipm}{u^{i,\pm}}
\newcommand{\fg}{{f}}
\newcommand{\ftd}{{\varphi}}
\newcommand{\ftn}{{\psi}}

\newcommand{\grt}{{G_M}}

\newcommand{\gre}[1]{{\Phi(\cdot;#1)}}
\newcommand{\greq}[1]{{\Phi_q(\cdot;#1)}}

\newcommand{\sepmgre}[1]{{\widehat{\Phi}^\pm(\cdot;#1)}}
\newcommand{\sepmgreq}[1]{{\widehat{\Phi}_q^\pm(\cdot;#1)}}
\newcommand{\sepgre}[1]{{\widehat{\Phi}^+(\cdot;#1)}}
\newcommand{\sepgreq}[1]{{\widehat{\Phi}^+_q(\cdot;#1)}}
\newcommand{\coefpmgre}[2]{{\widehat{\Phi}^\pm(#2;#1)}}
\newcommand{\coefpmgreq}[2]{{\widehat{\Phi}_q^\pm(#2;#1)}}

\newcommand{\omehm}{{\Omega^h_M}}

\newcommand{\gamhmp}{{\Gamma^h_M}}
\newcommand{\gamhmm}{{\Gamma^{-h}_M}}

\newcommand{\boul}{{M_L^-}}
\newcommand{\bour}{{M_L^+}}

\newcommand{\spahsm}[1]{{H^{#1}_{\sper}(\omehm)}} 
\newcommand{\spahsmloc}[1]{{H^{#1}_{\sper, \loc}(\Omega_M)}}
\newcommand{\spatracep}[1]{{H^{#1}_\sper(\gamhmp)}}
\newcommand{\spatracem}[1]{{H^{#1}_\sper(\gamhmm)}}

\newcommand{\raycoefp}[2]{{\widehat{#1}^+(#2)}}
\newcommand{\raycoefm}[2]{{\widehat{#1}^-(#2)}} 
\newcommand{\raycoefpm}[2]{{\widehat{#1}^\pm(#2)}}

\newcommand{\hdualcoefpm}[1]{{\widehat{#1}^\pm}}

\newcommand{\seth}{{H}}
\newcommand{\inc}{{\mathrm{inc}}}

\newcommand{\aher}{a}
\newcommand{\zg}{{z}}
\newcommand{\eg}{u}
\newcommand{\wgg}{{w}}

\newcommand{\vgg}{{v}}
\newcommand{\fug}{{{\varphi}}}
\newcommand{\xg}{{x}}

\newcommand{\Rang}{\mathcal{R}}

\newcommand{\Zd}{\Z^{d-1}}

\newcommand{\domp}{\omega}

\newcommand{\fugm}{{\varphi_1}}
\newcommand{\fuge}{{\varphi_2}}
\newcommand{\fgm}{f_1}
\newcommand{\fge}{{f_2}}

\newcommand{\ntm}{{\mu^{-1}}}

\newcommand{\nte}{{n}}

\newcommand{\qtm}{{q}}
\newcommand{\qte}{{p}}

\newcommand{\Htm}{{H}}
\newcommand{\Gtm}{{G}}
\newcommand{\Ntm}{{N}}
\newcommand{\Ttm}{{T}}
 
\newcommand{\RgHtm}{{{H}}}

\newcommand{\fundnp}[1]{{\Phi(n_p; #1)}}
\newcommand{\tfundnp}[1]{{\widetilde{\Phi}(#1)}}
\newcommand{\wpr}{{w}}
\newcommand{\twpr}{{\widetilde{\wpr}}}
\newcommand{\opes}[1]{{\widetilde{\mathcal{S}}_{#1}}}

\newcommand{\cdomp}{{\widetilde{\domp}}}

\newcommand{\fc}{{\widetilde{f}}}

\renewcommand{\qtm}{{Q}}
\renewcommand{\ntm}{{A}}
\newcommand{\nua}{{\nu_\ntm}}

\newcommand{\C}{{\mathbb{C}}}
\renewcommand{\Zd}{{\Z^{d-1}}}
\renewcommand{\Htm}{{\mathcal{H}}}
\newcommand{\Csp}{{\mathbf{H}(\Dint)}}
\newcommand{\bigped}{{\Theta}}

\renewcommand{\omehm}{{\bigped^h}}

\renewcommand{\Ttm}{{\mathrm{T}}}

\renewcommand{\ftd}{{g}}
\renewcommand{\ftn}{{h}} 
\newcommand{\ZM}{{\Z^{d-1}_{M}}}

\begin{document}

\title[Differential Imaging of Local Perturbations in Anisotropic  Periodic Media]{Differential Imaging of Local Perturbations in Anisotropic  Periodic Media}

\author{Thi-Phong Nguyen$^1$}

\address{$^1$ Department of Mathematics, Rutgers University, 110 Frelinghuysen Road,  Piscataway, NJ 08854-8019, USA.}
\ead{tn242@math.rutgers.edu}
\vspace{10pt}
\begin{indented}
\item[] 
\end{indented}

\begin{abstract}
We discuss the use of differential sampling method to image local perturbations in anisotropic periodic layers, extending earlier works on the isotropic case. We study in particular the new interior transmission problem that is associated with the inverse problem when only a single Floquet-Bloch mode is used. We prove Fredholm properties of this problem under similar assumptions as for classical interior transmission problems. The result of the analysis is then exploited to design an indicator function for the local perturbation. The resulting numerical algorithm is  validated for two dimensional numerical experiments with synthetic data. \end{abstract}

\section{Introduction }

We are interested in the imaging problem where one would like to identify the geometry of a local perturbation in a periodic media. We use multistatic measurements of scattered waves at a fixed frequency. This problem is related to applications in nondestructive testing of periodic structures which are of growing interest with the developments of sophisticated nano-structures like metamaterials, nanograss, etc.  In these applications, often, the  healthy periodic structure has complicated geometry and therefore one would like to avoid modeling issues associated with this background. It is therefore desirable to use an imaging method that does not rely on the Green function associated with the periodic background and directly provide an indicator function for the defect geometry. 
This is for example the case of the differential sampling method that was introduced in \cite{Thi-Phong3}, \cite{tpnguyen}, \cite{Thi-Phong4}. 
Our main objective here is to complement this literature by addressing the important case of possibly anisotropic background or defects.  

The imaging method developed in \cite{tpnguyen} is based on the generalized linear sampling method which was first introduced in  \cite{Audib2015a}, \cite{Audib2014} (see also \cite{CCH}). Sampling methods have been applied to the imaging of many periodic structure, see  \cite{Arens2010},  \cite{Arens2005},  \cite{Bourg2014}, \cite{Elsch2011a}, \cite{Lechl2013b}, \cite{armin},  \cite{nguye2012} for a sample of work. These works assume that the background Green function is computable.  In the case of our problem we do not make use of this Green function.  The main idea in the case of periodic background is to compare imaging functional associated with the full data with the imaging functional associated with single Floquet-Bloch data. The latter plays the role of data associated with a periodic background formed by the real background and the defect repeated periodically. This why our method can be compared to sampling methods using differential measurements as introduced in \cite{Audib2015}. Indeed in our case a single set of measurements is needed. 

The main ingredient in our analysis of the differential sampling method is the study of the new interior transmission problem that appear in the analysis of the single Floquet-Bloch mode sampling method. This problem couples the classical interior transmission problem with scattering problems associated with the other Floquet-Bloch modes. We prove Fredholm property of this problem using the T-coercivity approach \cite{Bonne2011a} and careful estimates on the exponential decay for wave solutions with imaginary wave numbers. As for classical interior transmission problems, the analysis of the anisotropic case is different from the isotropic case since the functional spaces are different. Our theoretical results only apply to the case where the Floquet-Bloch transform is reduced to a finite discrete sum. This corresponds to the case where the problem with defect is also a periodic problem with a different (larger) periodicity than the periodic background. 

Comparing sampling solutions associated with the periodic Green functions one
can design an indicator function of the defect geometry as in \cite{tpnguyen}.
The resulting algorithm is in fact independent from the assumption made in the
analysis on the periodicity of the problem with defect mentioned earlier. The
numerical  indicator function is tested and validated against synthetic
data. We discuss in particular the cases where the defects are inside one of the background inhomogeneous components and the case where it is not.

The paper is organized as follows. We first introduce the direct scattering
problem for anisotropic periodic layers and some key results on the varaitional
formultion and radiation conditions. The inverse problem is introduced in
Section 3 and the classical generalized sampling method is analysed for this
problem. We consider in Section 4 the inverse problem associated with a single
Floquet-Bloch mode and  introduce the new interior transmission problem that
shows up for the analysis of the method. Section 5 is dedicated to the analysis
of this new problem with the help of the T-coercivity approach. The last
section is dedicated to the numerical algorithm that allows us to identify the
geometry of the defect and some validating numerical results.

\section{The Direct Scattering Problem}
The scattering problem we are considering can be formulated in  $\R^d$, $d = 2$
or $d = 3$. A parameter  $L := (L_1, \ldots, L_{d - 1}) \in \R^{d-1}, \ L_j > 0, \
j = 1, \ldots, d-1$ will refer to the periodicity of the background with
respect to the first $d-1$ variables and we need to consider a second (artificial)
parameter $M := (M_1, \ldots, M_{d - 1}) \in
\N^{d-1}$ that refers to the number of periods in the truncated domain. A function defined in $\R^d$ is called
$L$ periodic if it is periodic with
period $L$ with respect to the $d-1$ first variables.    

We then consider the $ML-$periodic Helmholtz
equation (vector multiplications are to be understood component
wise, i.e. $ML =  (M_1L_1, \ldots, M_{d - 1}L_{d - 1})$) where the
  total field $u$ satisfies 
  \begin{equation}\label{C4Eq:Helmh}
	\left\{
	\begin{array}{lc}
		\nabla \cdot \ntm  \nabla u  + k^2 \nte u =  0  \quad \text{in} ~ \R^d \\[1.5ex]
		u ~ \text{is $ML-$periodic}
	\end{array}
	\right.
\end{equation}
and where  $k>0$ is the {\it wave number}.
%\begin{figure}[htpb]
%\centerline{\input{Structure/figInter.tex}}
%\caption{Sketch of the geometry for the $ML-$periodic problem}
%\label{Fig:ML-periodic}
%\end{figure}
We denote by  $D$ the support of $\ntm - I$ and $\nte -1$ which is assumed to be such that $\R^d \setminus D$ is connected; $\ntm$ is a $d \times d$ symmetric matrix with $W^{1,\infty}(\R^d)$-entries, $ML-$periodic and such that 
\[
	\ol{\xi} \cdot \Re(\ntm) \xi \geq a_0 |\xi|^2 \quad \text{and} \quad \ol{\xi} \cdot \Im(\ntm) \xi \leq 0
\] for all $\xi \in \C^d$ and some constant $a_0 > 0$. We further assume that the index of refraction {$\nte\in L^{\infty}(\R^d)$} is $ML-$periodic and satisfies $\Re(\nte)\geq \nte_0>0$,  $\Im(\nte)\geq 0$. Furthermore $\ntm = \ntm_p$ and $\nte = \nte_p$ outside a compact domain $\domp$ where $\ntm_p$ is a $d \times d$ matrix with $W^{1,\infty}(\R^d)$-entries and {$\nte_p \in L^\infty(\R^d)$ } such that $\ntm_p$ and $\nte_p$ are $L$-periodic. In addition there exists $h >0$ such that $\ntm  = I$, $\nte =1$ for $|x_d| > h$. Thanks to the $ML-$periodicity, solving equation \eqref{C4Eq:Helmh} in $\R^d$ is equivalent to solving it in the period 
 $$\bigped:=\bigcup_{m \in \Z^{d-1}_M} \Omega_m =  \interd{\boul}{\bour} \times \R $$
  with $ \Omega_m : =
  \interd{-\frac{L}{2} +mL}{ \frac{L}{2} +mL}\times\R$,  $\boul:= \left(\left\lfloor-\frac{M}{2}\right\rfloor + \frac{1}{2}\right)L$, $\bour: =\left(\left\lfloor\frac{M}{2}\right\rfloor + \frac{1}{2}\right) L$, \;  and \; $ \Z^{d-1}_M: = \{ m\in \Zd, \textstyle \left\lfloor-\frac{M_\ell}{2}\right\rfloor+1 \leq m_\ell \leq \left\lfloor\frac{M_\ell}{2}\right\rfloor , \; \ell = 1, \ldots, d-1\}$, where we use the notation $\interd{a}{b} := [a_1, b_1] \times \cdots \times
[a_{d-1}, b_{d-1}]$ and $\lfloor\cdot \rfloor$ denotes the floor function. We also shall use the notation $\normM{a} := |a_1\cdot a_2 \cdots a_{d-1}|$. By the definition of $\Omega_m$, we also have  $\Omega_m = \Omega_0 + mL$. {Without loss of generality we assume that the local perturbation $\domp$ is located in only one period, say $\Omega_0$ (i.e $ m = 0$)}. We call $D_p$ the support of $\ntm_p - I$ and $\nte_p-1$. This implies $D = D_p \cup \domp$ and note that $\ntm = I$ and $\nte=1$ outside $D$. 

We  consider down-to-up or up-to-down  incident plane waves of the form
\begin{equation}\label{inc}
	 u^{i,\pm}(x,j)=\frac{-\i}{2\,\olbetm{j}}e^{\i \alpm{j} \ol{x} \pm  \i \olbetm{j} x_d } 
\end{equation}
where 
$$\quad \textstyle{ \alpm{j} := \frac{2\pi}{ML}j \quad \mbox{ and } \quad  \betm{j} := \sqrt{k^2 - \alpha^2_\sper(j)}}, \quad \Im( \betm{j}) \ge 0, \quad j \in \Zd$$
and $x=(\ol{x},x_d) \in \R^d$. Then the scattered field $u^s = u- u^i$ verifies 
\begin{equation}\label{scat}
	\left\{
	\begin{array}{lc}
		\nabla \cdot A \nabla u^s + k^2 nu^s =  - \nabla
		 \cdot \qtm \nabla u^i  - k^2 \qte u^i \quad \text{in} ~ \R^d, \\[1.5ex]
		u^s ~ \text{is $ML-$periodic}
	\end{array}
	\right.
\end{equation} where  $\qtm$ and $\qte$ are the contrasts defined by 
\[
	\qtm: = \ntm - I \quad \text{and} \quad \qte: = \nte - 1, 
\] $I$ is the $3 \times 3$ identity matrix. {To  ensure that the scattered wave is outgoing, we impose as a radiation condition the Rayleigh expansion}
\begin{equation} \label{C4:RDC}
\left\{
\begin{array}{lc}
	u^s(\ol{x},x_d) = \sum_{\ell \in \Zd} \raycoefp{u^s}{\ell} e^{\i(\alpm{\ell} \ol{x} + \betm{\ell} (x_d-h))},  \quad \forall \  x_d >h, \\[1.5ex]
	u^s(\ol{x},x_d) = \sum_{\ell \in \Zd} \raycoefm{u^s}{\ell} e^{\i(\alpm{\ell}\ol{x} - \betm{\ell} (x_d+h))}, \quad \forall \ x_d <-h,
\end{array}
\right.
\end{equation}
where the Rayleigh coefficients $\raycoefpm{u^s}{\ell}$ are given by
\begin{equation}\label{RayleighCh2}
\left.
\begin{array}{l l}
\raycoefp{u^s}{\ell} := \dsp \frac{1}{|\interd{\boul}{\bour}|} \int_{\interd{\boul}{\bour}} u^{s}(\ol{x}, h) e^{ - \i \alpm{\ell} \cdot \ol{x}} \d{\ol{x}},  \\[2ex]
\raycoefm{u^s}{\ell}:= \dsp \frac{1}{|\interd{\boul}{\bour}|} \int_{\interd{\boul}{\bour}} u^{s}(\ol{x}, - h) e^{ - \i \alpm{\ell} \cdot \ol{x}} \d{\ol{x}}.
\end{array} 
\right.	
\end{equation}
We shall use the notation
 $$\omehm := \interd{\boul}{\bour} \times ]-h,h[$$
$$\gamhmp:= \interd{\boul}{\bour} \times \{h\}, \quad   \gamhmm:=
\interd{\boul}{\bour} \times \{- h\}.$$
For {an} integer $m$, we denote by $H^m_{\sper}(\bigped^h)$ the restrictions to $\bigped^h$ of
functions  that are in $H^m_{\mathrm{loc}}(|x_d| \le h)$ and are
$ML-$periodic. The space $\spatracep{1/2}$ is then defined as the space of
traces on $\gamhmp$ of functions in $H^1_\sper(\bigped^h)$ and the space
$\spatracep{-1/2}$ is defined as the dual of $\spatracep{1/2}$. Similar definitions are used for $\spatracem{\pm 1/2}$. Using the radiation condition \eqref{C4:RDC} we can define the Dirichlet-to-Neumann operators $T^\pm: H^{1/2}_{\#}(\Gamma^{\pm h}_M) \to H^{-1/2}_{\#}(\Gamma^{\pm h}_M)$ as 
\begin{equation}
\label{Def:DtN}
 \phi \mapsto  T^{\pm} \phi: = \i \sum_{\ell \in \Zd} \betm{\ell} \raycoefpm{\phi}{\ell} e^{\i \alpm{\ell} \cdot \ol{x}}
\end{equation}
%\begin{equation}
%\label{Def:DtN}
%	\begin{split}
%		\opedtn^+ : \spatracep{1/2} &\longrightarrow \spatracep{-1/2} \\
%		\phi &\longmapsto \opedtn^+ \phi =  \i \sum_{\ell \in \Zd} \betm{\ell} \raycoefp{\phi}{\ell} e^{\i \alpm{\ell} \cdot \ol{x}}\\[1.5ex]
%		\opedtn^- : \spatracem{1/2} &\longrightarrow \spatracem{-1/2} \\
%		\phi  &\longmapsto \opedtn^- \phi =  \i \sum_{\ell \in \Zd} \betm{\ell} \raycoefm{\phi}{\ell} e^{\i \alpm{\ell} \cdot \ol{x}}
%	\end{split}
%\end{equation}
More generally for a given  $f = (\ftm, \fte) \in L^2(\Omega^h_M)^d \times L^2(\Omega^h_M)$, we  consider the following problem: Find $w \in \spahsm{1}$ satisfying 
\begin{equation}
\label{eq:w}
\nabla \cdot \ntm \nabla w + k^2 \nte w = - \nabla \cdot
 \qtm \ftm - k^2 \qte \fte
\end{equation}
together with the Rayleigh radiation condition (\ref{C4:RDC}). Then we make the following assumption:
\begin{asp} \label{Ass:nk}
The parameters $\ntm$, $\nte$ and the wave-number $k> 0$ are such that  \eqref{eq:w}  with $\ntm$, $\nte$ and  with $\ntm$, $\nte$ replaced by  $\ntm_p$, $\nte_p$ are both well-posed for all $ f = \f \in L^2(\bigped^h)^d \times L^2(\bigped^h) $. 
\end{asp}
\noindent
We remark that the solution  $w \in \spahsm{1}$ of \eqref{eq:w} can be
extended to a function in $\bigped$ satisfying $\nabla \cdot \ntm \nabla \wgg + k^2 \nte \wgg =
- \nabla \cdot \qtm \ftm - k^2\qte \fte$ {in $\R^d$}, using the Rayleigh expansion \eqref{C4:RDC}. We denote by $\grt(x)$ the $ML-$ periodic Green function satisfying $ \Delta \grt + k^2 \grt = - \delta_0$ in $\bigped$ and the Rayleigh radiation condition. Then $w$ has the representation  
\begin{eqnarray}
\label{eq:2forw}
	w(x) &=& \nabla \cdot  \int_{D} \grt(x-y)  \qtm(y) \big( \nabla w +  \ftm \big)(y)   \d{y}   \nonumber \\
	&+& k^2 \int_{D} \grt(x - y)  \qte(y) \big(w +   \fte \big)(y)  \d{y}
\end{eqnarray}

\noindent
 Let $z \in \R^d$ be an arbitrary point, we set {$\Phi(\cdot;z) = \grt(\cdot - z)$ } and recall that it  can be expressed as 
\begin{equation} \label{form:GM}
	{\Phi(x;z) } = \frac{\i}{2ML} \sum_{\ell \in \Z} \frac{1}{\betm{\ell}} e^{\i \alpm{\ell} \ol{(x - z)} + \i\betm{\ell} |x_d - z_d|}.
\end{equation} For latter use, we denote by $\sepmgre{z}: = \{ \coefpmgre{z}{\ell}\}_{\ell \in \Zd}$  the Rayleigh sequences of $\Phi(\cdot,z)$, where the Rayleigh coefficient $\coefpmgre{z}{\ell}$ { is } given by 
\begin{equation}\label{lunch}
	\coefpmgre{z}{\ell}:= \textstyle{\frac{\i}{2\normM{ML} \beta_\sper(\ell)}} e^{-\i (\alpm{\ell} \ol{\zg} - \betm{\ell}|\zg_d \mp h|)} .
\end{equation}

\section{The Inverse Problem}
%The inversion method is based on the so-called the Generalized Linear Sampling Method, which was first introduced in  \cite{Audib2015a}, \cite{Audib2014} (see also  \cite[Chapter 2]{CCH}), augmented with the idea of {differential imaging} introduced in  \cite{Audib2015} which was  adapted to this problem in \cite{tpnguyen}.
%
\noindent 
As described above we have two choices of interrogating waves. If  we use  down-to-up (scaled) incident plane waves $\uip(x;j)$ defined by (\ref{inc}), then our measurements (data for the inverse problem) are  given by the Rayleigh sequences 
$$
\raycoefp{u^s}{\ell;j},  \quad  (j, \ell) \in \Zd \times \Zd,  
$$
whereas if we use up-to-down (scaled) incident plane waves $\uim(x;j)$ defined by (\ref{inc}) then our  measurements are given the Rayleigh sequences
$$
\raycoefm{u^s}{\ell;j}, \quad  (j, \ell) \in \Zd \times \Zd.
$$
These measurements  define  the so-called near field (or data) operator which is used to derive the indicator function of the defect. More specifically, let us consider the (Herglotz) operators $\Htm^{\pm}: \ell^2(\Zd) \rightarrow L^2(D)^d \times L^2(D)$ defined by 
\begin{equation} \label{defH}
\Htm^{\pm}\aher:=  \Big(\sum_{j \in \Z} \aher(j) \nabla \uipm(\cdot; j)\big|_D, \sum_{j \in \Z} \aher(j) \uipm(\cdot; j)\big|_D\Big), \quad \forall \, \aher = \{\aher(j)\}_{j \in \Z} \in \ell^2(\Zd). 
\end{equation}
Then $\opeh^{\pm}$ is compact, injective (will be proved later) and its adjoint $\Odual{\opeh^{\pm}}: L^2(D)^d \times L^2(D) \to \ell^2(\Zd) $ is given by 
 \begin{equation}\label{adjointH}
\Odual{\Htm^{\pm}} \fug := \{\hdualcoefpm{\fug}(j)\}_{j \in \Z},  \quad \forall \; \fug = (\fugm, \fuge) \in L^2(D)^d \times L^2(D),
\end{equation} where
\[
 \quad \hdualcoefpm{\fug}_j : = \int_{D} \Big( \fugm(\xg)\cdot \nabla  \ol{\uipm(\cdot; j)}(\xg) + \fuge(\xg)\cdot  \ol{\uipm(\cdot; j)}(\xg)\Big) \d{\xg}.
\]
Let us denote by  $\seth_{\inc}^{\pm}(D)$ the closure of the range of $\opeh^{\pm}$ in $L^2(D)^d \times L^2(D)$. We then consider the (compact) operator $\opeg^{\pm}:\seth_{\inc}^{\pm}(D) \rightarrow \ell^2(\Zd)$ defined by
\begin{equation} \label{defG}
{\opeg^{\pm}} (f) := \{\raycoefpm{w}{\ell}\}_{\ell \in\Zd},
\end{equation}
where $\{\raycoefpm{w}{\ell}\}_{\ell \in\Zd}$ is the Rayleigh sequence of $ \wgg \in\spahsm{1}$ the solution of \eqref{eq:w}. We now define the sampling operators $\open^{\pm}: \ell^2(\Zd) \rightarrow \ell^2(\Zd) $ by
\begin{equation}
\open^{\pm} (\aher) = {\opeg^{\pm}}  \, \opeh^\pm(\aher).
\end{equation}
By linearity of the operators ${\opeg^{\pm}}$ and $\opeh^\pm$ we also get an equivalent definition of $\open^{\pm}$ directly in terms of measurements as
\begin{equation}
[\open^{\pm} (\aher)]_\ell = \sum_{j \in \Zd} a(j) \, \raycoefpm{u^s}{\ell;j} \quad \ell \in \Zd.
\end{equation}
Let us introduce the operator $\Ttm: L^2(D)^d \times L^2(D) \rightarrow L^2(D)^d \times L^2(D)$ defined by 
\begin{equation}  \label{defT}
	\Ttm \fg:= \Big( - \qtm (\fgm + \nabla \wgg|_D), \, k^2 \qte(\fge + \wgg|_{D} \Big), \quad \forall \fg = (\fgm,\fge) \in L^2(D)^d \times L^2(D)
\end{equation}
with $\wgg$ being the solution of \eqref{eq:w}.  We then have the following: 
\begin{lemma}
\label{Ch4:Lem:defG}
	The operators $\Gtm^{\pm}$ defined by \eqref{defG} can be factorized as 
	\[
		\Gtm^{\pm} = \Odual{\Htm^{\pm}} \Ttm.
	\]	
\end{lemma}

\begin{proof} 
Let $\fg = (\fgm, \fge) \in L^2(D)^d \times L^2(D)$ and $\wgg \in \spahsm{1}$ be solution to \eqref{eq:w}. Let us write $T_1(f) := - \qtm (\fgm + \nabla \wgg|_D)$ and $T_2(f): = k^2 \qte(\fge + \wgg|_{D})$.  Then, by definition of the Rayleigh coefficients and combining with the representation of $\grt$ in \eqref{form:GM} and the writing of $\wgg$ as in \eqref{eq:2forw} we  have
\begin{multline*}
	\raycoefpm{\wgg}{j} =  \frac{1}{2ML}\int_{x_d=\pm h} e^{-\i \alpm{j} \cdot\ol{x}} \int_D \sum_{\ell \in \Z} \frac{\alpm{\ell}}{\betm{\ell}} e^{\i \alpm{\ell} \cdot (\ol{x} - \ol{y}) \, + \, \i\betm{\ell} |h \mp y_d|} \cdot(T_1(\fg)(y)) \d{y} \d{\ol{x}} \\ 
	+  \frac{\i}{2ML}\int_{x_d=\pm h} e^{-\i \alpm{j} \ol{x}} \int_D \sum_{\ell \in \Z} \frac{1}{\betm{\ell}} e^{\i \alpm{\ell} (\ol{x} - \ol{y}) \, + \, \i\betm{\ell} |h \mp y_d|} T_2(\fg)(y) \d{y} \d{\ol{x}} \\
 =  \int_D \frac{\alpm e^{\i \betm{j} h}}{2\, \betm{j}}  e^{-\i \alpm{j}\cdot \ol{y} \, \mp \, \i \betm{j} y_d} \cdot T_1(\fg)(y)\d{y}  + \int_D \frac{\i e^{\i \betm{j} h}}{2\, \betm{j}}  e^{-\i \alpm{j}\ol{y} \, \mp \, \i \betm{j} y_d} T_2\fg(y)\d{y}
\end{multline*}
Observing  that 
$$ \frac{\alpm e^{\i \betm{j} h}}{2\,\betm{j}}  e^{-\i \alpm{j}\cdot ol{y} \, \mp \, \i \betm{j} y_d} = \nabla \ol{\uipm}(y;j) \text{ and } \frac{\i e^{\i \betm{j} h}}{2\,\betm{j}}  e^{-\i \alpm{j}y_1 \, \mp \, \i \betm{j} y_2} = \ol{\uipm}(y;j)$$  we then have 
$$\raycoefpm{\wgg}{j} = \int_D \Ttm_1 \fg(y)\cdot \nabla \ol{\uipm}(y;j) +   \Ttm_2 \fg(y) \ol{\uipm}(y;j) \d y,$$ which proves the lemma.
\end{proof}
 The following properties of  $\opeg^{\pm}$ and and $\opeh^\pm$ are crucial to our inversion method.  To state them, we must recall the standard  {\it interior transmission problem}:
 $({\eg}, \vgg) \in
H^1(D) \times H^1(D)$ such that 
\begin{equation} \label{oitp}
\left\{ \begin{array}{lll}
\nabla \cdot (\ntm \nabla {\eg}) + k^2 \nte {\eg} = 0 \quad& \mbox{ in } \; D, 
\\[6pt]
\Delta \vgg + k^2 \vgg = 0  \quad &\mbox{ in } \; D,
\\[6pt]
 {\eg} - \vgg= {\ftd} \quad &\mbox{ on } \; \partial D,
\\[6pt]
\dsp \frac{\partial {\eg}}{\partial \nua} -  \frac{\partial \vgg}{\partial \nu} = {\ftn} \quad &\mbox{ on } \; \partial D,
\end{array}\right.
\end{equation}
for given $({\ftd},\ftn) \in H^{1/2}(\partial D) \times H^{-1/2}(\partial D) $
where $\nu$ denotes the outward normal on $\partial D$ and $\partial u/\partial \nua$ denotes the co-normal derivative, i.e
\[
	\frac{\partial u}{\partial \nua} = \nu \cdot \ntm \nabla u.
\]
 Values of $k$ for which
this problem with $\ftd = 0$ and $\ftn = 0$ has non-trivial solution  are referred to as {\it transmission
eigenvalues}. For our purpose we shall assume that this problem is well posed.  Up-to-date results on this problem can be found in \cite[Chapter 3]{CCH} where in particular one finds sufficient solvability conditions. In the sequel we make the following assumption. 
\begin{asp} \label{HypoLSM}
{$\partial D \cap \partial \Omega_0 =\emptyset$} and the refractive indexes $\ntm$, $\nte$ and the wave number  $k>0$ are such that \eqref{oitp}
has a unique solution.
\end{asp}
\noindent
%Remark that we here specified that the solution of the interior transmission
%problem is in $H^1(D) \times H^1(D)$ which is the natural space of solutions in
%the cases $\ntm \neq I$ in a neighborhood
%of $\partial D$ \cite{Bonne2011a}. 
%In particular,  if $\Re(n-1)>0$ or $-1<\Re(n-1)<0$ uniformly in a neighborhood of $\partial D$  inside $D$  the interior transmission problem (\ref{oitp}) satisfies the Fredholm alternative, and the  set of real standard transmission eigenvalues is discrete (possibly empty).  Thus Assumption \ref{HypoLSM} holds  as long as $k>0$ is not a transmission eigenvalue.

\subsection{Some key properties of the introduced operators}
We still keep the assumption (that is not essential but simplifies some of the
arguments, and justifies the use of $\Ntm^{+}$ or $\Ntm^{-}$ and not both of
them) 
$$
\bigped \setminus D \mbox{ is connected.}
$$
{In order to avoid repetitions and since the main novelty is in the study of the case of single Floquet Bloch mode, we hereafter indicate without proofs the main properties of the operators $\Htm^\pm$, $G^\pm$ and $T$. These properties can be proved in very similar way as in \cite{} and following the adaptations for periodic probels as in \cite{}. We will prove similar properties for the case of single Floquet-Bloch mode operators and the reader can easily adapt those proofs to the easier case here}
The first step towards the justification of the sampling methods is the characterization of the closure of the range of $\Htm^\pm$.
\begin{lemma} \label{Ch4:lemHerg} 
The operator $\Htm^\pm$ is compact and injective. Let  $\RgHtm_{\inc}^{\pm}(D)$ be  the closure of the range of $\Htm^\pm$ in $L^2(D)^d \times L^2(D)$. Then 
\begin{equation} \label{Ch4:forHinc}
\RgHtm_{\inc}^{\pm}(D) = \{(\fugm,\fuge) = (\nabla \vgg, \vgg)| \;\; v \in H^1(D) \; : \; \Delta \vgg
+ k^2 \vgg = 0 \mbox{ in } D\}.
\end{equation}
Assume that Assumptions \ref{Ass:nk} and \ref{HypoLSM} hold. Then the
operator ${\Gtm^{\pm}}:  \RgHtm_{\inc}(D) \rightarrow \ell^2(\Z)$ defined by (\ref{defG}) is injective with dense range.
\end{lemma}
\begin{proof}
The compactness and the injectivity of the operators $\Htm^\pm$ and the operators $\Gtm^{\pm}$ directly follow from Lemma 3.3 and Lemma 3.5 in \cite{Thi-Phong3}. 
\end{proof}

\noindent  Let $q$ be a fixed parameter in $\ZM$, we denote by $\greq{z}$ the {outgoing fundamental solution that verifies} 
\begin{equation}\label{phiq}
\Delta \greq{z} + k^2 \greq{z} =  -\delta_{z}\qquad \mbox{in} \;\Omega_0
\end{equation}
 and which is $\alpha_q$ quasi-periodic with period $L$ {with $\alpha_q :=  2\pi  q/ (ML)$}. Then $ \greq{z}$ has the expansion 
  \begin{equation} \label{form:phiq}
	\greq{z} = \frac{\i}{2ML} \sum_{\ell \in \Z} \frac{1}{\betm{ q+ M\ell}} e^{\i \alpm{a + M\ell} \ol{(x - z)} + \i\betm{q+M\ell} |x_d - z_d|}.
\end{equation}The Rayleigh coefficients $\sepmgreq{z}$ of $\greq{z}$ are given by 
\begin{equation}\label{hhh}
	\hspace*{-0.5cm}\coefpmgreq{z}{j} =  \left \{
\begin{array}{cl} 
   \textstyle{\frac{\i}{2\normM{L} \betm{q + M\,\ell}}} e^{-\i (\alpm{q + M\,\ell} \ol{\zg} - \betm{q + M\,\ell}|\zg_d \mp h|)}  & \, \mbox{if} \, j = q + M \ell, \; \ell \in \Zd, \\[1.5ex]
 0 &\, \mbox{if} \, j \neq q + M \ell, \;  \ell \in \Zd.
\end{array}
 \right.
\end{equation}

\noindent We now  prove one of the main ingredients for the justification of
the inversion methods discussed below. 

\begin{theorem} \label{Ch4:TheoG}
 For $z \in \R^d$, $\sepmgre{z}$
belongs to $\Rang(\Gtm^{\pm})$ if and only if $\zg \in D$ and $\sepmgreq{z}$
belongs to $\Rang(\opeg^{\pm})$ if and only if $\zg \in D_p$, where $q$ is a
fixed parameter in $\Z_M$.
\end{theorem}

\begin{proof}
We now prove that $\sepmgre{z}$
belongs to $\Rang(\Gtm^{\pm})$ if and only if $\zg \in D$. We first observe that
 $\sepgre{z}$ is the Rayleigh sequence of
  $\gre{z}$ satisfying $\Delta {\gre{z}} + k^2 {\gre{z}}= -
  \delta_z \mbox{ in } \bigped$ and the Rayleigh radiation condition. Let $\zg \in D$. We
  consider $(\eg, \vgg) \in H^1(D) \times  H^1(D) $ as being  the 
solution to \eqref{oitp} with 
\begin{equation} \label{Ani:continuity}
\ftd(\xg) = \Phi(\xg;z)  \mbox{ and } \ftn(\xg) =
\partial \Phi(\xg;z)/\partial \nu(\xg) \;\; \mbox{ for } \xg \in \partial D.
\end{equation}
We then define $\wgg$ by 
$$\begin{array}{lcl}
\wgg(\xg) = &\eg(\xg) - \vgg(\xg) & \mbox{in } D, \vspace{2mm}\\
\wgg(\xg) =  &\Phi(\xg;z)  & \mbox{in } \bigped \setminus D.
\end{array}
$$
Due to (\ref{Ani:continuity}), we have that $\wgg \in \spahsmloc{1}$
and satisfies (\ref{eq:w}). Hence $\Gtm^{+} \vgg = \sepgre{z}$. 

\vspace{0.05in}

\noindent Now let $\zg \in \bigped \setminus D$. Assume that there exists $\fug = (\nabla \fg, \fg)\in \RgHtm_{\inc}(D)$
such that $\Gtm^{+} \fug =\sepgre{z}$. This implies that $\wgg =  \gre{z}$ in $\{x \in \bigped, \pm x_d \geq h \}$ where
$\wgg$ is the solution to \eqref{eq:w}. By the unique continuation principle we deduce that $\wgg =  \gre{z}$ in $\bigped \setminus D$ . This gives a contradiction since   $\wgg
\in H^1_{\sper,\loc}(\bigped \setminus D) $ while $\gre{z} \notin
H^1_{\sper,\loc}(\bigped \setminus D) $.

\medskip

\noindent The proof of the statement $\sepmgreq{z}$
belongs to $\Rang(\opeg^{\pm})$ if and only if $\zg \in D_p$ follows the same
lines as above replacing $\Phi(\cdot; z)$ by  $\Phi_q(\cdot; z)$. The reader
can also refer to the proof of Lemma 4.7 in \cite{Thi-Phong3}.
\end{proof}

\begin{lemma} \label{Ch4:Lem:T}
Assume that Assumptions 1 and 2 hold.  Then the operator  $\Ttm: L^2(D)^d \times L^2(D) \to L^2(D)^d \times L^2(D)$ defined by \eqref{defT} satisfies 
\begin{equation}
 \Im \left( \Ttm \phi, \phi \right) \geq 0, \quad \forall \phi \in \RgHtm_{\inc}(D).
\end{equation}
 Assume
in addition that $\xi \cdot \qtm \xi \geq \sigma_{\nte} > |\xi|^2 \, \text{ in} \; D$ (respectively $ - \xi \cdot \qtm \xi \geq \sigma_{\nte} > |\xi|^2 \, \text{ in} \; D$) and $k$ is not a transmission eigenvalue. Then $ - \Re \Ttm = \Ttm_0 + \Ttm_1$, where $\Ttm_0$ (respectively $-\Ttm_0$) is self-adjoint and coercive and $\Ttm_1$ is compact on $\RgHtm_{\inc}(D)$. Moreover, $\Ttm$ is injective on $\RgHtm_{\inc}(D)$.
\end{lemma}

\begin{proof}
 Let $\fug = (\fugm,\fuge) \in L^2(D)^d \times L^2(D)$ and $\wgg_\fug$ be
 solution to \eqref{eq:w} associated with $ \fg = \fug$.  By definition of
 the operator $\Ttm$ we have 
\begin{multline} \label{Ch4:lasttpn1}
	 \left( \Ttm \fug, \fug \right)_{L^2(D)^d\times L^2(D)} =    \int_{D}  - \qtm (\fugm + \nabla \wgg_\fug) \cdot \ol{\fugm} + k^2 \qte(\fuge + \wgg_{\fug}) \ol{\fuge}  \d{\xg}\\
	 = -  \int_{D} \Big(\qtm |\fugm+ \nabla \wgg_\fug|^2  - k^2 \qte|\fuge + \wgg_\fug|^2 \Big) \d{\xg} \\+  \int_{D} \Big(\qtm(\nabla \wgg_\fug + \fugm) \cdot \nabla \ol{\wgg_\fug}   - k^2 \qte (\fuge + \wgg_\fug) \ol{\wgg_\fug} \Big)\d{\xg}.
\end{multline} Integrating $\int_{D} \qtm(\nabla \wgg_\fug + \fugm) \cdot \nabla \ol{\wgg_\fug}   - \qte (\fuge + \wgg_\fug) \ol{\wgg_\fug} \d{\xg}$ by part and by writing $ \Delta \wgg_{\fug} + k^2 \wgg_{\fug} = -\nabla \cdot \qtm (\fugm+ \nabla \wgg_\fug) - k^2 \qte (\fuge+ \wgg_\fug)$ we have 
\begin{multline} \label{Ch4:lasttpn2}
	   \int_D \qtm (\nabla \wgg_\fug + \fugm) \cdot \nabla \ol{\wgg_\fug}   - \qte (\wgg_{\fug} + \fuge) \wgg_{\fug} \d{\xg} \\ = \langle T^+ {\wgg_\fug},  {\wgg_\fug}\rangle +  \langle T^- {\wgg_\fug}, {\wgg_\fug} \rangle  - \int_{\bigped^h} |\nabla \wgg_\fug|^2 - k^2  |\wgg_\fug|^2 \d{\xg},
\end{multline}
where $T^{\pm}$ be the Dirichlet-to-Neumann  operators defined in \eqref{Def:DtN}.  Therefore, substituting \eqref{Ch4:lasttpn2} into \eqref{Ch4:lasttpn1} we end
up with:
\begin{multline} \label{Ch4:lasttpn3}
	 \left( \Ttm \fug, \fug \right)_{L^2(D)^d \times L^2(D)} =  \int_{D} - \qtm |\fugm + \nabla \wgg_\fug|^2 + k^2 \qte|\fuge + \wgg_\fug|^2  \d{\xg} \\
	  - \int_{\bigped^h} (|\nabla \wgg_\fug|^2 - k^2  |\wgg_\fug|^2) +  \langle T^+{\wgg_\fug}, {\wgg_\fug}\rangle  + \langle T^- {\wgg_\fug}, {\wgg_\fug} \rangle
\end{multline}
Thanks to the non-negative sign of the imaginary part of $T^\pm$ and the assumption $\Im (\nte) \geq 0$  we deduce that
\begin{equation*}
	\Im  \left( \Ttm \fug, \fug \right) =  \int_D \Im(\nte)|\fuge +
        \wgg_\fug|^2 \d{\xg} +  \Im  \langle T^{+}{\wgg_\fug},
        {\wgg_\fug}\rangle  + \langle T^- {\wgg_\fug}, {\wgg_\fug} \rangle \geq 0. 
\end{equation*}
For the case $ \qtm$ positive definite on $D$ one can define $T_0 : L^2(D)^d \times L^2(D) \to L^2(D)^d \times L^2(D)$ by
$$
\left( T_0 \fug, \psi \right)_{L^2(D)^d \times L^2(D)} := \int_{D} \qtm (\fugm + \nabla
\wgg_\fug)\cdot \overline{ (\psi_1 + \nabla
\wgg_\psi)} + \fuge \overline \psi_2 \d{\xg}  + \int_{\bigped^h} (\nabla
\wgg_\fug \cdot \overline{ \nabla \wgg_\psi} )\d{\xg} 
$$
which is indeed a selfadjoint and coercive operator. Using \eqref{Ch4:lasttpn3} one then
deduces that $-T +T_0 : \RgHtm_{\inc}(D) \to L^2(D)^d \times L^2(D)$ is compact by the $H^2$ regularity outside
$D$ of $\wgg_\fug$ and the Rellich compact embedding theorem. Observe that we
used that the operator is restricted to $\RgHtm_{\inc}(D) $ to infer
compactness of the terms involving $\fug_2$ in the expression of $(-T
+T_0)(\fug)$.

\medskip

\noindent For the case $\qtm$ negative definite on $D$ we first observe that
\eqref{Ch4:lasttpn1} and \eqref{Ch4:lasttpn2} also lead to
\begin{multline} \label{Ch4:lasttpn5}
\left( \Ttm \fug, \fug \right)_{L^2(D)^d \times L^2(D)} =  \int_{D}  -\qtm
|\fugm |^2 + \int_{\bigped^h}  \ntm |\nabla \wgg_\fug|^2 + 2i \int_{D}
-\qtm \Im( \nabla \wgg_\fug \cdot \fugm ) \\ + \int_{D} k^2 \qte(\fuge +
\wgg_\fug) (\overline{\fuge} -
\overline{\wgg_\fug})  \d{\xg}  
	  - \int_{\bigped^h} k^2  |\wgg_\fug|^2 \d{\xg}  -  \langle
          T^+{\wgg_\fug}, {\wgg_\fug}\rangle  - \langle T^- {\wgg_\fug},
          {\wgg_\fug} \rangle.
\end{multline}
We then define $T_0 : L^2(D)^d \times L^2(D) \to L^2(D)^d \times L^2(D)$ by
$$
\left( T_0 \fug, \psi \right)_{L^2(D)^d \times L^2(D)} := \int_{D} - \qtm\fugm
\overline{\psi_1}  + \fuge \overline \psi_2 \d{\xg}  + \int_{\bigped^h} \ntm (\nabla
\wgg_\fug \cdot \overline{ \nabla \wgg_\psi} )\d{\xg}
$$
which is also selfadjoint and coercive. Using \eqref{Ch4:lasttpn5} one 
deduces using the same arguments as in the previous case that $T -T_0 :
\RgHtm_{\inc}(D) \to L^2(D)^d \times L^2(D)$ is compact.

 \medskip
 
\noindent In the case $k$ is not a transmission eigenvalue, the injectivity of $\Ttm^+$ is implied for instance by Assumption \ref{HypoLSM} and the
factorization $\Gtm^{+} = \Odual{\Htm^{+}} \, \Ttm$: Assume that $\fug =
(\nabla \fg, \fg)\in  \RgHtm_{\inc}(D)$ and $T\fug = \Big(-\qtm(\nabla \fg +
\nabla \wgg_\fug) , \, k^2 \qte (\fg + w_\fug) \Big) = 0$. This implies, using
the factorization $\Gtm^{+} = \Odual{\Htm^{+}} \, \Ttm$ that
$\raycoefp{w_\fug}{j} =  0$ for all $j \in \Z$ and therefore $w_\fug = 0$ in
$\bigped \setminus D$ (by unique continuation {principle}). With $\fug = (\nabla
\fg,\fg) \in \RgHtm_{\inc}(D)$ and $\fg$ verifying $\Delta \fg + k^2 \fg = 0$
in $D$ we get that $u: = \fg + w_\fug$ and $v:=\fg$ satisfying the interior
transmission problem \eqref{oitp} with $\varphi = \psi = 0$. We then deduce
that $u = v = 0$. This proves the injectivity of the operator $\Ttm$. 
\end{proof}
\noindent
Another main ingredient is a symmetric factorization of  an appropriate operator given in terms of $\open^{\pm}$. To this end, for a generic operator $F:H\to H$, where $H$ is a Hilbert space, with adjoint $F^*$ we define
\begin{equation}\label{sharp}
 	\opef_{\sharp} := |\Re(\opef) |+ |\Im(\opef) |
 \end{equation}
 where \; $\Re(\opef) := \frac{1}{2}
\left(\opef +\Odual{\opef}\right)$, \; $\Im(\opef) := \frac{1}{2\i}
\left(\opef - \Odual{\opef}\right)$.  We then have the following: 
\begin{theorem} \label{TheoFactorization}
	Assume that the hypothesis of Lemma \ref{Ch4:Lem:T} hold true. Then we have the following factorization 
	\begin{equation} \label{titifactN}
		\open_{\sharp}^{\pm} = \Odual{\opeh^{\pm}} \, \opet_{\sharp} \, \opeh^\pm,
	\end{equation} where $\opet_{\sharp}: L^2(D) \to L^2(D)$ is self-adjoint and coercive on $\seth_{\inc}(D)$.  
\end{theorem}

\noindent For latter use, we assume that each period of $D_p$ is composed by $J \in \N$ disconnected components and the defect $\domp$ may contain or have non-empty intersection with at least one component (recall that $\domp$ assume to be located in $\Omega_0$). For convenience, we now introduce some additional notations.  We denote by $\Dpint$ the union of the components of $D_p \cap \Omega_0$ that have nonempty intersection with $\domp$,  and by $\Dpout$ its complement in $D_p\cap\Omega_0$, i.e the union of all the components of $D_p \cap \Omega_0$ that do not intersect $\omega$. Furthermore, we denote by $\Dint: = \Dpint \cup \domp$ and by $\Dtot : = \Dint \cup \Dpout$. Obviously,  $\Dtot = D \cap \Omega_0$. (see  Fig. \ref{Fig:ML-periodic} and note that if $\domp$ does not intersect with $D_p$ then $\Dpint \equiv \emptyset$, $\Dpout \equiv D_p \cap \Omega_0$ and $\Dint = \domp$). We consider the following $ML$-periodic copies of the aforementioned regions
\begin{equation}
\label{nota}	
 \Dpoutp =  \bigcup_{m \in \Z_M} \Dpout + mL,   \quad  \Dint_p := \bigcup_{m \in \Z_M} \Dint + mL  \quad \text{and}  \quad  \Dtot_p := \bigcup_{m \in \Z_M} \Dtot + mL
\end{equation} Remark that $\Dtot_p \equiv D_p \cup \big(\cup_{m \in \Z_M} \domp + mL\big)$ contains $D$ and the $L$-periodic copies of $\domp \setminus D_p$. {We remark that $n = n_p = 1$ in $\Dtot_p \setminus D$.}
\begin{figure}[htpb]
\centerline{\begin{tikzpicture}[scale=0.9]
\draw[blue,thick,dashed] (-6.5,1.9) -- (10.5,1.9);
\draw[blue,thick,dashed] (-6.5,-1.7) -- (10.5,-1.7);
\draw[blue,thick] (-6,-3.2) -- (-6,2.2); 
\draw[blue,thick] (-2,-2) -- (-2,2.2);
\draw[blue,thick] (2,-2) -- (2,2.2); 
\draw[blue,thick] (6,-2) -- (6,2.2);
\draw[blue,thick] (10,-3.2) -- (10,2.2);

%\fill[color=red!60, fill=red!50, very thick](-4,-0.2) circle (0.5);
%\fill[color=red!60, fill=red!50, very thick](0,-0.2) circle (0.5);
%\fill[color=red!60, fill=red!50, very thick](4,-0.2) circle (0.5);
%\fill[color=red!60, fill=red!50, very thick](8,-0.2) circle (0.5);

\fill[color=red!60, fill=red!50, very thick](-4.,-0.7) rectangle (-3.,0.3);
\fill[color=red!60, fill=red!50, very thick](0,-0.7) rectangle (1.,0.3);
\fill[color=red!60, fill=red!50, very thick](4.,-0.7) rectangle (5.,0.3);
\fill[color=red!60, fill=red!50, very thick](8.,-0.7) rectangle (9.,0.3); 

\fill[color=red!60, fill=red!50, very thick](-5.,0.9) circle (0.75);
\fill[color=red!60, fill=red!50, very thick](-1.,0.9) circle (0.75);
\fill[color=red!60, fill=red!50, very thick](3.,0.9) circle (0.75);
\fill[color=red!60, fill=red!50, very thick](7.,0.9) circle (0.75);

\fill[color=blue!60, fill=blue!40, very thick](-0.2,1.2) circle (0.4);
%\fill[color=blue!60, fill=blue!60, very thick](-1.4,1.2) circle (0.2);
%\fill[color=blue!60, fill=blue!60, very thick](-1,1.2) circle (0.2);

\small
\draw[blue] (0,2.3) node {$ n = n_p = 1$};
\draw[blue] (-6.4,1.4) node {$ h$};
\draw[blue] (-6.6, -1.2) node {$ -h$};
%\draw[blue] (1.6,-1.2) node {$\Omega_0$};
\draw[black] (-0.2,1.2) node {$ \domp$};
%\draw[blue] (-4,-0.2) node {$ D_p$};
%\draw[blue](0,-0.2) node {$ D_p$};
%\draw[blue] (4,-0.2)  node {$ D_p$};
%\draw[blue] (8,-0.2) node {$ D_p$};
%\draw[blue] (-5.3,0.5) node {$ D_p^{j_0}$};
\draw[blue](-1.1,0.7)node {$\Dpint$};
%\draw[blue] (2.7,0.5) node {$ D_p^{j_0}$};
%\draw[blue] (6.7,0.5) node {$ D_p^{j_0}$};

\draw[blue] (0.4,-0.3) node {$ \Dpout$}; 

\draw[blue] (1.6,-1.4) node {$ \Omega_0$};

\draw [blue,ultra thin,<->] (-2,-1.9) -- (2,-1.9);  \draw[blue,thick,dashed] (0,-2.1) node {$L$};

\draw[black] (- 0.6,-2.7) node {$ \Dint: = \Dpint \cup \domp$,}; 

\draw[black] (2.6,-2.7) node {$ \Dtot: = \Dint \cup \Dpout$}; 

\draw [red,thin,<->] (-6,-3.2) -- (10,-3.2);  \draw[red,thick,dashed] (0,-3.6) node {$ ML$};

\end{tikzpicture}}
\caption{Sketch of the geometry for the $ML-$periodic problem with the notations.}
\label{Fig:ML-periodic}
\end{figure}
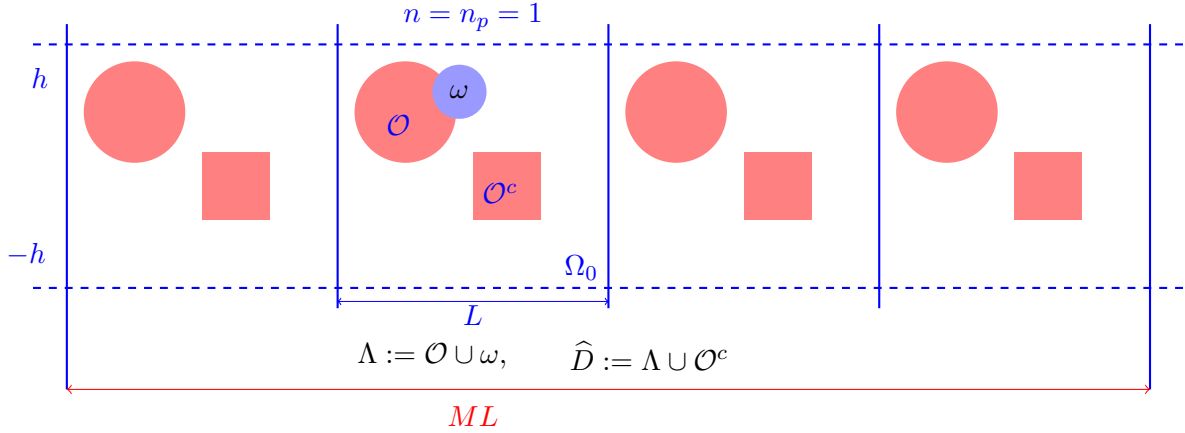

\section{The Near Field Operator for a Single Floquet-Bloch Mode}

\noindent
Our goal is to derive an imaging method that resolves only $\omega$ without knowing or recovering $D_p$. This leads us to introducing the sampling operator for a single Floquet-Bloch mode whose analysis will bring up a new interior transmission problem. We start with the definition of a quasi-periodic function. 
\begin{definition}
 A function $u$ is called quasi-periodic with parameter $\xi = (\xi_1, \cdots ,\xi_{d - 1})$ and period $L = (L_1, \cdots , L_{d-1})$, with respect to the first $d-1$ variables (briefly denoted as $\xi-$quasi-periodic with period $L$) if:
\begin{equation*}
	u(\ol{x} + (jL), x_d) = e^{i\xi \cdot (jL)} u(\ol{x}, x_d), \quad \forall j \in \Z^{d-1}.
\end{equation*} 
\end{definition}
%
% Let $q$ be a fixed parameter in $\ZM$, we denote by $\greq{z}$ the {outgoing fundamental solution that verifies} 
%\begin{equation}\label{phiq}
%\Delta \greq{z} + k^2 \greq{z} =  -\delta_{z}\qquad \mbox{in} \;\Omega_0
%\end{equation}
% and which is $\alpha_q$ quasi-periodic with period $L$ {with $\alpha_q :=  2\pi  q/ (ML)$}. The Rayleigh coefficients $\sepmgreq{z}$ of $\greq{z}$ are given by 
%\begin{equation}\label{hhh}
%	\hspace*{-0.5cm}\coefpmgreq{z}{j} =  \left \{
%\begin{array}{cl} 
%   \textstyle{\frac{\i}{2\normM{L} \betm{q + M\,\ell}}} e^{-\i (\alpm{q + M\,\ell} \ol{\zg} - \betm{q + M\,\ell}|\zg_d \mp h|)}  & \, \mbox{if} \, j = q + M \ell, \; \ell \in \Zd, \\[1.5ex]
% 0 &\, \mbox{if} \, j \neq q + M \ell, \;  \ell \in \Zd.
%\end{array}
% \right.
%\end{equation}
%
\noindent Let $\aher \in \ell^2(\Zd)$, we define for $q \in
\Z^{d-1}_M$, the element $a_q \in \ell^2(\Zd)$ by
\[ 
	\aher_q(j) := \aher(q + jM).
\]
We then define the operator $\opei_q : \ell^2(\Zd) \to \ell^2(\Zd)$, which transforms $\aher \in \ell^2(\Zd)$ to $\tilde \aher \in \ell^2(\Zd)$ such that 
\begin{equation}\label{iop}
\tilde{\aher}_q = \aher\quad \mbox{and} \quad \tilde{\aher}_{q'} = 0 \; \mbox{ if } \; q \neq q'. 
\end{equation}
We remark that $\opei^{*}_q (\aher) = \aher_q$, where $\opei^{*}_q: \ell^2(\Zd)
\to \ell^2(\Zd)$ is the dual of the operator $\opei_q$. 
The single Floquet-Bloch mode Herglotz operator $\opeh^{\pm}_q: \ell^2(\Zd) \to L^2(D)$
is defined by 
\begin{equation} \label{2defHq}
	\opeh^{\pm}_q \aher := \opeh^{\pm} \opei_q \aher = \sum_{j} \aher(j) \uipm(\cdot; q + jM)|_{D}
\end{equation}
and the  single Floquet-Bloch mode  near field (or data) operator $\open^{\pm}_q:
\ell^2(\Zd) \to \ell^2(\Zd)$  is defined by 
\begin{equation} \label{2defNq} 
	\open^{\pm}_q \,\aher  = \opei^{*}_{q} \, \open^{\pm} \,  \opei_q \, \aher.
\end{equation} 
We remark that $\opeh^{\pm}_q \aher$ is an $\alpha_q-$quasi-periodic function
with period $L$. The sequence $\open^{\pm}_q \,\aher $ corresponds to the
Fourier coefficients of the $\alpha_q-$quasi-periodic component of the
scattered field in the decomposition \eqref{for:decompw}. This operator is then somehow
associated with $\alpha_q-$quasi-periodicity. One immediately sees from the factorization $\open^{\pm} = (\opeh^{\pm})^{*} \,  \opet \,
\opeh^{\pm}$ that the following factorization holds.
\begin{equation} \label{2FactfNq} 
	\open^{\pm}_q = (\opeh^{\pm}_q)^{*} \, \opet \, \opeh^{\pm}_q.
\end{equation}
For later use we also define the operator  $\opeg^{\pm}_q: \ol{\Rang(\opeh^{\pm}_q)} \to \ell^2(\Zd)$  by 
\begin{equation} \label{2defGq} 
	 \opeg^{\pm}_q = (\opeh^{\pm}_q)^{*} \opet |_{\ol{\Rang(\opeh^{\pm}_q)} }
\end{equation} where the operator $\opet$ is defined by \eqref{defT}. 
\begin{lemma}
\label{lemHergq}
The operator $\opeh^{\pm}_q$ is injective and 
\begin{multline*}
\ol{\Rang(\opeh^{\pm}_q)} = \seth^q_{\inc}(D) := \big\{(\fugm,\fuge) = (\nabla \vgg, \vgg)| \;\; v \in H^1(D) \; : \;   \Delta v + k^2 v = 0 \; \mbox{in } D \;  \mbox{and} \\
\; v|_{D_p} \; \mbox{is $\alpha_q-$quasi-periodic}\big\}.
\end{multline*}
\end{lemma}
\begin{proof} 
$\Htm^{\pm}_q$ is injective since $\Htm^{\pm}$ is injective and $\opei_q$ is injective. We now prove that $(\Htm^{\pm}_q)^{*}$ is injective on $\RgHtm^q_{\inc}(D)$. Let $\fug = (\nabla \fg,\fg) \in \RgHtm^q_{\inc}(D)$ and assume $(\Htm^{\pm}_q)^{*} (\varphi) = 0$. We define
\[
		u(x): = \frac{1}{M} \nabla \cdot \int_{D} \Phi_q(x - y) \big(-\nabla{\fg}(y)\d{y}\big) + \frac{1}{M}\int_{D} \Phi_q(x - y) \fg(y)\d{y}. 
\] From the expansion of $\Phi_q(x)$ as in \eqref{form:phiq}  and using the same calculations as in the proof of Lemma \ref{Ch4:Lem:defG} we have that $  \raycoefpm{u}{j} = 0$ for all  $j \neq q + M\ell$ and $ \raycoefpm{\eg}{q+M\ell} = \left(\Odual{\Htm^{\pm}}(\varphi)\right)(q+M\ell) =((\Htm^{\pm}_q)^{*} (\varphi))(\ell) = 0$. Therefore $u$ has all Rayleigh coefficients equal 0, which implies that 
\[ 
	u = 0, \quad \; \mbox{for} \; \pm x_d > h.
\]
We now observe that  for all $y \in D$, $\Delta \Phi_q(\cdot ; y) +
k^2 \Phi_q(\cdot ; y) = 0$ in the complement of $ \Dtot_p$. This implies that
$$\Delta u + k^2 u = 0 \quad \text{in } \; \R^d \setminus \Dtot_p. $$ 
Using a unique continuation argument we infer that $ u = 0$ in $\bigped
\setminus \Dtot_p$. Therefore, $u \in H^1_0(\Dtot_p)$ by the
regularity of volume potentials. We now consider two cases:

\bigskip

\noindent{\it If \; $\domp \subset D_p$}, then $\Dtot_p \equiv D_p$, i.e $u \in H^1_0(D_p)$. Moreover, by definition, $u$ verifies $\Delta u + k^2 u = \Delta f -f$ in $D_p$. Since $u \in H^1_0(D_p)$ and $\Delta f + k^2f = 0$ in $D_p$, we then have
\begin{equation}
0 = \int_{D_p} ( \Delta u +k^2 u) \ol{f} \d{x}  = \int_{D_p} ( -\Delta f + k^2 f) \ol{f} \d{x}  = \int_{D_p} ( k^2 +1)|f|^2 \d{x}
\end{equation}
This proves that $f = 0$, which yields the injectivity of $(\opeh^{\pm}_q)^{*}$ on $\seth^q_{\inc}(D)$.

\medskip

\noindent{\it If \; $\domp \not \subset D_p$}, let denote by  $\cdomp: = \domp \setminus D_p$ then $\cdomp \neq \emptyset$. 
%We then rewrite $u(x)$ as 
%\begin{align}
%		u(x) = &\frac{1}{M}\nabla \cdot \int_{\cdomp} \Phi_q(x ; y) (-\nabla \fg(y))\d{y} + \frac{1}{M}\int_{\cdomp} \Phi_q(x ; y) \fg(y)\d{y} \nonumber \\
%                & + \frac{1}{M} \nabla \cdot \int_{D_p } \Phi_q(x ; y) (-\nabla \fg(y))\d{y}  
%                +  \frac{1}{M} \int_{D_p} \Phi_q(x ; y) \fg(y)\d{y} 
%\end{align}
Since $\varphi|_{D_p}$ and $\Phi_q$ are $\alpha_q-$quasi-periodic
functions with period $L$, we then have for $x \in D_p \cap \Omega_m$.
\begin{align}
\label{eq:2DpOm}
		u(x) = &\frac{1}{\normM{M}}\nabla \cdot \int_{\cdomp} \Phi_q(x ; y) (-\nabla \fg(y))\d{y} + \frac{1}{\normM{M}}\int_{\cdomp} \Phi_q(x ; y) \fg(y)\d{y} \nonumber \\
                & + \nabla \cdot \int_{D_p \cap \Omega_m} \Phi_q(x ; y) (-\nabla \fg(y))\d{y}  
                + \int_{D_p \cap \Omega_m} \Phi_q(x ; y) \fg(y)\d{y} 
\end{align}
Recall that $\Delta \Phi_q(\cdot ; y) + k^2 \Phi_q(\cdot; y) = -\delta_y$ in
$D_p \cap \Omega_m$ and $\Delta \Phi_q(\cdot ; y) + k^2 \Phi_q(\cdot; y) = 0$ in
$\cdomp$, we then obtain from \eqref{eq:2DpOm} that  for $m\in \Z^{d-1}_{M}$,
\begin{equation}\label{Ch4:fintas1}
\Delta u(x) + k^2 u(x) =  \Delta \fg(x) -  \fg(x) \mbox{ in } D_p \cap \Omega_m. 
\end{equation}
Let us set for $x \in
\cdomp + mL$, $m\in \ZM:$
$$ \fg_m(x) := e^{\i \alpha_q \cdot mL} \varphi(x - mL).$$ Using the  $\alpha_q-$quasi-periodicity of $\Phi_q$, we have for $x
                \in \cdomp+ mL$,
\begin{align*}
u(x): = &\frac{1}{\normM{M}}\nabla \cdot\int_{\cdomp+ mL} \Phi_q(x ; y) (-\nabla \fg_m(y))\d{y} +\frac{1}{\normM{M}}\int_{\cdomp+ mL} \Phi_q(x ; y) \fg_m(y)\d{y}  \\
                &+\frac{1}{\normM{M}}
                \nabla \cdot \int_{D_p} \Phi_q(x ; y) (-\nabla  \fg(y))\d{y}  +\frac{1}{\normM{M}} \int_{D_p} \Phi_q(x ; y) \fg(y)\d{y} .
\end{align*} Moreover, in this case\, $\Delta \Phi_q(\cdot ; y) + k^2 \Phi_q(\cdot; y) = -\delta_y$ in $ \cdomp + mL$ 
 and $\Delta \Phi_q(\cdot ; y) + k^2 \Phi_q(\cdot; y) = 0$ in  $D_p \cap \Omega_m$ then
\begin{equation}\label{Ch4:fintas2}
	\Delta u(x) + k^2 u(x) = \Delta \fg_m  -  \fg_m  \mbox{ in } \cdomp + mL.
\end{equation} We now define the function $\widetilde \fg \in H^2(\Dtot_p)$ by 
$$
\widetilde \fg = \fg \mbox{ in } D_p \mbox{ and }\widetilde \fg =
\fg_m \mbox{ in }  \cdomp + mL, \, m \in \ZM.$$
 Then $\widetilde \fg$ satisfies  \; $\Delta \widetilde \fg + k^2 \widetilde \fg = 0 \mbox{ in }\Dtot_p.$ Since  $u \in H^1_0(\Dtot_p)$ then according to  \eqref{Ch4:fintas1} and \eqref{Ch4:fintas2}  we have 
\begin{align*}
0 = 	\int_{\Dtot_p} (\Delta u + k^2 u)  \ol{\widetilde \fg} &= \int_{D_p}\big( \Delta  \fg - \fg\big) \ol{\fg} \d{x} +  M \int_{\cdomp}\big(\Delta \fg -  \fg\big) \ol{\fg}  \\
&=  \int_{D_p} ( k^2 +1)|f|^2 \d{x} + M\int_{\cdomp} ( k^2 +1)|f|^2 \d{x} 
\end{align*}
(remind that $f = \fc$ in $D$), which implies $\fg = 0$ in $D$. This proves the injectivity of $\Odual{\Htm^{\pm}}$ on $\RgHtm^q_{\inc}(D)$ and hence proves the Lemma.
\end{proof}

\noindent 
We now see that $ \varphi(j;\ol{x}): =e^{\i \alpm{j} \ol{x} }  = e^{\frac{2\pi}{ML}j \ol{x}},  j\in \Z $ is a Fourier basic of $ML$ periodic function in $L^2(\bigped)$, for that any  $w \in L^2(\bigped)$ which is $ML$ periodic, has the expansion 
\begin{equation}
	w(x) = \sum_{j \in \Z}\widehat{w}(j,x_d) \varphi(j;\ol{x}), \quad \text{where} \quad \widehat{w}(j,x_d) := \frac{1}{\normM{ML}} \int_{\bigped} w(x) \ol{\varphi(j;\ol{x})} \d{\ol{x}}.
\end{equation} Splitting index $j$ by module $M$ as $j = q + M\ell$, for $q \in \ZM$ and $\ell \in \Z$, and  then arranging the previous sum of $w$, we obtain a finite sum with respect to $q$,
\begin{equation}
\label{for:decompw}
w = \sum_{q\in \Z_M} w_q,
\end{equation} where $w_q: = \sum_{\ell \in \Z} \widehat{w}(q+M\ell,x_d) \varphi(q+M\ell; \ol{x})$ is $\alpha_q-$quasi-periodic with period $L$, here $\alpha_q: = \frac{2\pi}{L}q$. Thus any $ML-$periodic  function $w \in L^2(\bigped)$ can be decomposed 
 where $w_q$ is $\alpha_q-$quasi-periodic with period $L$. Moreover, by the orthogonality of the Fourier basic $\{ \varphi(j; \cdot) \}_{j \in \Z}$, we have that 
\begin{equation}
\raycoefpm{w_q}{j} = 0 \quad \text{if} \; j \neq q+ M\ell, \; \ell \in \Z \qquad \text{and} \qquad \raycoefpm{w}{q+M\ell} = \raycoefpm{w_q}{q+M\ell}
\end{equation} where $\raycoefpm{w_q}{j}$ the Rayleigh sequence of $w_q$ defined in \eqref{RayleighCh2}. Coming back to the definition of $\opeg^{\pm}_q$, we see that $\opeg^{\pm}_q(f)$ is a Rayleigh sequence of $\raycoefpm{w}{j}$ at all indices $j = q + M\ell, \, \ell \in \Z$, where $w$ is solution of \eqref{eq:w}. Seeing also the line above that theses coefficients come from the Rayleigh sequence of $w_q$ where $w_q$ is one of the component of $w$ using the decomposition \eqref{for:decompw}, which is $\alpha_q-$ quasi periodic.  Let $\fug:=(\fugm, \fuge) = (\nabla f, f) \in \seth^q_{\inc}(D)$, we then introduce the $\alpha_q-$quasi-periodic function $\fugc: = (\nabla \fc, \fc)$ where $\fc$ is given by  
\begin{equation}
\label{def:ftilde}
	\fc: =  \left\{\begin{array}{ll}
		f &\quad \text{in} \quad \bigped \setminus \Dint_p  \\ [1.5ex]
		e^{i \alpha_q mL} f|_{\Dint} &\quad \text{in} \quad \Dpint + mL, \quad \forall \; m\in \Z_M.
\end{array}	
\right.	
\end{equation} then $f$ and $\fc$ (respectively $\fug$ and $\fugc$) coincide in $D$. Therefore equation \eqref{eq:w}  with data  $\fug = (\nabla f, f) \in \seth^q_{\inc}(D) $ is equivalent to 
\begin{equation}
\label{rewrite:eqw}
	\nabla \cdot \ntm \nabla w + k^2 \nte w = - \nabla \cdot
 \qtm \nabla \fc - k^2 \qte \fc
\end{equation}
Using the decomposition \eqref{for:decompw} for $w$, and that fact that $n_p$ and $\ntm_p$ are periodic, $\varphi$ is $\alpha_q-$quasi-periodic and $n - n_p$ and $\ntm - \ntm_p$ are compactly supported in one period $\Omega_0$, equation \eqref{rewrite:eqw} becomes
\begin{equation*} 
	\nabla \cdot \ntm_p \nabla  w_q + k^2n_p w_q  = \nabla \cdot (\ntm_p - \ntm) \nabla w + k^2(n_p - n) w  - \nabla \cdot \qtm \nabla \fc -  k^2 \qte \fc \; \;  \text{in} \; \; \Omega_0. 
\end{equation*} Denoting by $\widetilde{w} : = w - w_q$, the previous equation is equivalent to 
\begin{equation} 
\label{eq:wq}
	\nabla \cdot \ntm w_q + k^2 \nte w_q  = \nabla \cdot (\ntm_p - \ntm) \nabla \widetilde{w} + k^2(n_p - n) \widetilde{w}  - \nabla \cdot \qtm \nabla \fc -  k^2 \qte \fc \; \; \text{in} \;\; \Omega_0.
\end{equation} 
Therefore, operator $\opeg_q^{\pm}:\ol{\Rang(\opeh^{\pm}_q)} \rightarrow \ell^2(\Zd)$ can be equivalently defined as
\begin{equation} \label{defGqequiv}
\opeg^{\pm}_q (f) := \opei^{*}_q\{\raycoefpm{w_q}{\ell}\}_{\ell \in\Zd},
\end{equation} where $w_q$ solution of \eqref{eq:wq} and $w_q + \widetilde{w}$ is solution of \eqref{eq:w}. 

\medskip

\noindent Central to the analysis of the sampling method for a single Floquet-Bloch mode $q$ is the following new interior transmission problem. 
\begin{definition}[The new interior transmission problem]\label{nitp}{\it Find $(u, f) \in
H^1(\Dint) \times H^1(\Dint)$ such that 
\begin{equation} \label{eqnew}
\left\{ \begin{array}{llll}
\nabla \cdot \ntm \nabla  {\eg} + k^2 n {\eg} - \nabla \cdot (\ntm_p - \ntm) \nabla \opes{k}(f)  - k^2(n_p - n) \opes{k}(f) = 0 \quad &\mbox{ in } \; \Dint, 
\\[6pt]
\Delta f + k^2 f = 0  \quad &\mbox{ in } \; \Dint,
\\[6pt]
 {\eg} - f= {\ftd} \quad &\mbox{ on } \; \partial \Dint,
\\[6pt]
{\big(\ntm \nabla  {\eg} - (\ntm_p - \ntm) \nabla \opes{k}(f) - \nabla f\big) \cdot \nu} ={\ftn} \quad &\mbox{ on } \; \partial \Dint, 
\end{array}\right.
\end{equation}
for given $({\ftd},\ftn) \in H^{1/2}(\partial \Dint) \times H^{-1/2}(\partial \Dint) $
where  \;  $\opes{k}: H^1( \Dint) \to H^1(\Dint)$ is defined by
\begin{equation}
\label{def:2Stilde}
\opes{k}(f) :=  \nabla \cdot  \int_{ \Dint} \tfundnp{x,y} \big((\ntm_p - I)\nabla f  \big)(y)  \d{y} +  k^2 \int_{ \Dint} \tfundnp{x,y} \big( ( \nte_p - 1) f  \big) (y) \d{y}, 
\end{equation} with the kernel 
$$\tfundnp{x,y}: = \sum_{0 \neq m \in \Z_M} e^{\i \alpha_q m L} \fundnp{x-mL-y}$$
 and $\fundnp{\cdot}$ is the $ML$-periodic outgoing fundamental solution that verifies 
 \begin{equation}\label{phinp}
	\nabla \cdot \ntm_p \nabla \fundnp{\cdot}  + k^2 n_p \fundnp{\cdot} = - \delta_0 \mbox{ in } \bigped
\end{equation}
 and where  $\nu$ denotes the unit normal on  $\partial \Dint$ outward to $\Dint$. 
}
\end{definition}
The analysis requires that this problem is well posed. We make this as an assumption here and we shall provide in the following section sufficient conditions on the coefficients $A_p$ and $n_p$ that ensure this assumption. 
\noindent
\begin{asp} \label{as-nitp} The parameters  $\ntm$, $\nte$ and $k>0$   are such that  the new interior transmission problem defined in Definition \ref{nitp} has a unique solution.
\end{asp}
The form of the new transmission eigenvalue problem shows up when we treat the injectivity of the operator $G_q^\pm$ as shown in the proof of the following result.

\begin{theorem}
\label{Lem:injectiveG_q} Suppose that Assumptions \ref{Ass:nk}, \ref{HypoLSM} and \ref{as-nitp} hold. Then the operator $\opeg^{\pm}_q : \seth^q_{\inc}(D) \to \ell^2(\Zd)$ is injective with dense range. 
\end{theorem}
\begin{proof}
Assume that $\fug = (\nabla \fg, \fg) \in \seth^q_{\inc}(D)$ such that $\opeg_q(\fug) = 0$. Let $w$ be solution of \eqref{eq:w} with data $\fug$. From \eqref{defGqequiv} we have that the Rayleigh sequence of $w_q$ vanishes, where $w_q$ is the $\alpha_q-$quasi-periodic component obtained from the decomposition of $w$ as in \eqref{for:decompw}, and verifies 
\begin{equation} \label{eq:2w1.1}
	\nabla \cdot \ntm_p \nabla  w_q + k^2n_p w_q  = \nabla \cdot (\ntm_p - \ntm) \nabla w + k^2(n_p - n) w  - \nabla \cdot \qtm \nabla \fc -  k^2 \qte \fc \; \;  \text{in} \; \; {\Omega_0},
\end{equation}  where $\fc$ is defined in \eqref{def:ftilde}.  By unique continuation argument as at the beginning of the
proof of Lemma \ref{lemHergq} we deduce that
\begin{equation}
\label{wzero1}
	w_q = 0 \quad \mbox{in} \quad \bigped \setminus \Dtot_p.
\end{equation}  This deduces that
\begin{equation}
\label{bdrycond}
 w_q = 0 \mbox{ and } \nu \cdot \ntm_p \nabla w_q = \nu \cdot \big( (\ntm_p - \ntm) \nabla w - \qtm \nabla \fc\big) \quad \mbox{on } \partial \Dtot_p.
 \end{equation}
We also observe that $\fc$ verifies 
\begin{equation}
\label{eq:fc}
	\Delta \fc + k^2 \fc = 0 \qquad \text{in} \quad \Dtot_p.
\end{equation} By the $\alpha_q$-quasi-periodicity of $w_q$ and $\fc$, it is sufficient to prove that $\fc = 0$ in $\Omega_0$. 
\\
In the domain $\Dpout$, $n = n_p$, {$A=A_p$} and $\Dpout \cap \Dint = \emptyset$. Then  $w_q$ and $\fc$ verifies
\begin{equation}
\label{eq:2itpwq1}
\left \{
	\begin{array}{lc}
		\nabla \cdot \ntm \nabla w_q + k^2n w_q  =  - \nabla \cdot \qtm \nabla \fc - k^2 \qte \fc & \text{in} \quad   \Dpout,\\[1.5ex]
		\Delta \fc + k^2 \fc  = 0 & \text{in} \quad   \Dpout.
		\end{array}
\right.
\end{equation} Combine with \eqref{wzero1}, we then obtain that $(w_q + \fc, \fc) \in H^1(\Dpout) \times H^1(\Dpout)$ and verifies equation \eqref{oitp} with the homogeneous boundary condition. Therefore, Assumption \ref{HypoLSM}  implies that $w_q + \fc = \fc = 0$ in $\Dpout$. This is equivalent to 
 \[
	w_q  = \fc = 0 \quad \text{in} \quad \Dpout.
\] 
 We now prove that $\fc = 0$ in $ \Dint$. We first {express the quantity } $w - w_q $ in terms of $\fc$ using the property that $\fc = 0$ outside $\Dint$. To this end, recalling that  $\fc = f$ in $D$, we can  write \eqref{eq:w}  in terms of $\fc$ as
\begin{equation} 
	\nabla \cdot \ntm_p \nabla w + k^2 n_p w = \nabla \cdot (\ntm_p - \ntm) \nabla w +  k^2 (n_p - n) w  - \nabla \qtm \nabla \fc -  k^2 \qte \fc
\end{equation}
and then have
\begin{eqnarray}
	w(x) &=& - \nabla \cdot \int_{D} \Big((\ntm_p - \ntm) \nabla w - \qtm \nabla \fc \Big)(y) \fundnp{x-y}  \d{y}   \nonumber \\
	&-& k^2 \int_{D}  \Big((\nte_p - \nte)w -  \qte \fc \Big)(y) \fundnp{x-y}  \d{y}
\end{eqnarray}
Using the facts that $\fc = 0$ and $n = n_p$ in  $\Dpoutp$, i.e. $n_p = n = 1$ in $\Dint_p \setminus D$ we have 
\begin{eqnarray}
\label{eq:reforw}
	w(x)&=& \nabla \cdot \int_{\Dint_p \setminus \Dint} (\ntm_p -I) \nabla \fc (y) \fundnp{x-y} \d{y} \nonumber \\
	&+&  k^2 \int_{\Dint_p \setminus \Dint} (\nte_p -1) \fc (y) \fundnp{x-y} \d{y} \nonumber \\
	&-& \nabla \cdot \int_{\Dint} \Big((\ntm_p - \ntm) \nabla w - \qtm \nabla \fc \Big)(y) \fundnp{x-y} \d{y} \nonumber \\
	&-& k^2 \int_{\Dint} \Big((\nte_p - \nte)w -\qte \fc \Big)(y) \fundnp{x-y} \d{y}
\end{eqnarray}
From \eqref{bdrycond}, we deduce that for all  $\theta \in H^1(\Dint)$ such that $\nabla \cdot {\ntm_p} \nabla \theta  + k^2 \nte_p \theta = 0$ we have 
\begin{equation}
\int_{\Dint} \Big(\nabla \cdot {\ntm_p}  \nabla w_q  + k^2 \nte_p w_q\Big) \ol\theta = \int_{\partial \Dint} \nu \cdot \big( (\ntm_p - \ntm) \nabla w - \qtm \nabla \fc\big) \ol{\theta} \,ds,
\end{equation}
implying from \eqref{eq:2w1.1} that 
\begin{multline}
\int_{\Dint} \nabla \cdot \Big( (\ntm_p - \ntm) \nabla w  - \qtm \nabla \fc \Big) \ol\theta \d{x}  + k^2 \int_{\Dint} \Big((\nte_p - \nte)  w  - \qte \fc  \Big) \ol{\theta} \d{x}=  \\ \int_{\partial \Dint} \nu \cdot \big( (\ntm_p - \ntm) \nabla w - \qtm \nabla \fc\big) \ol{\theta} \,ds.
\end{multline} This is equivalent to 
\begin{multline}
\label{for:tgan1}
- \int_{\Dint} \Big( (\ntm_p - \ntm) \nabla w  - \qtm \nabla \fc \Big) \cdot \nabla \ol\theta \d{x}  + k^2 \int_{\Dint} \Big((\nte_p - \nte)  w  - \qte \fc  \Big) \ol{\theta} \d{x}= 0
\end{multline}
Remark that for $x \notin \Dint$, $\nabla \cdot \ntm_p \nabla \fundnp{x-y} + k^2 n_p \fundnp{x-y} = 0$ for all $ y\in\Dint$.  Applying \eqref{for:tgan1} to $\theta(y): = \fundnp{x-y}$ we have 
\begin{equation*}
 - \int_{\Dint} \Big( (\ntm_p - \ntm) \nabla w  - \qtm \nabla \fc \Big) \cdot \nabla_y    \ol{\fundnp{x-y}} \d{y} + k^2 \int_{\Dint} \Big((\nte_p - \nte)  w  - \qte \fc  \Big)  \ol{\fundnp{x-y}} \d{y} = 0
\end{equation*} This is equivalent to 
\begin{equation}
	  \nabla \cdot \int_{\Dint} \Big((\ntm_p - \ntm) \nabla w - \qtm \nabla \fc \Big) \fundnp{x-y} \d{y} \nonumber + k^2\int_{\Dint_p } \Big(( \nte_p - \nte) w - p\fc \Big) \fundnp{x - y}  \d{y} = 0
\end{equation}
Combined with $\fc = 0$ outside $\Dint_p$, we then conclude from \eqref{eq:reforw} that 
\begin{eqnarray}
	w(x)  =& \nabla& \cdot \int_{\Dint_p \setminus \Dint} \fundnp{x - y} (\ntm_p - I) \nabla \fc  \d{y}   \nonumber\\
	& +&  k^2\int_{\Dint_p \setminus \Dint} (\nte_p -1) \fc (y) \fundnp{x-y} \d{y} \quad \mbox{ for } x \notin \Dint. 
\end{eqnarray}
Next we define 
\begin{eqnarray}
\label{def:twpr}
	\twpr(x) =& \nabla& \cdot \int_{\Dint_p \setminus \Dint} \fundnp{x - y} (\ntm_p - I) \nabla \fc  \d{y}   \nonumber\\
	& +& k^2 \int_{\Dint_p \setminus \Dint}  \fundnp{x-y} ( \nte_p - 1) \fc (y) \d{y} \quad \, x \in \bigped.
\end{eqnarray}
We observe that $ \nabla \cdot \ntm_p \nabla  \twpr + k^2 n_p \twpr  = 0$ in $\Dint$. We now keep $w$ and $w_q$ as above and let  $\widehat{w}: = w_q + \twpr$ in $\Dint$ which obviously verifies 
\begin{equation} \label{reform-nitp0}
	\nabla \cdot \ntm \nabla  \widehat{w} + k^2n \widehat{w}  =  - \nabla \cdot \qtm \nabla \fc  - k^2 \qte \fc \quad \text{in} \; \Dint.
\end{equation}
 By Assumption \ref{Ass:nk} we have, from uniqueness of solutions to the ML-periodic scattering problem, that  $w = \widehat{w}$ in $\Dint$.  This proves in particular that $\twpr = w - w_q$ in $\Dint$. Noticing that 
$$
\twpr |_\Dint = \opes{k}(f),
$$
we then can reformulate \eqref{reform-nitp0} as 
\begin{equation}
\label{eq:injitp}
		\begin{array}{lll}
		\nabla \cdot \ntm \nabla w_q + k^2n w_q  &=  \nabla \cdot (\ntm_p - \ntm) \nabla \opes{k}(f) \quad &\\[1.25ex]
		& \quad  + k^2(\nte_p - \nte) \opes{k}(f)  - \nabla \cdot \qtm \nabla f - k^2 \qte f  & \; \; \text{in} \; \; \Dint.
		\end{array} 
\end{equation} 
Combining \eqref{eq:injitp} and \eqref{bdrycond}  we see that the couple $u:=w_q+f$ and $f$ verifies the homogeneous version of the new interior transmission problem \eqref{eqnew}. Assumption \ref{as-nitp} now implies that $ f = 0$ in $\Dint$, which proves the injectivity of $G_q$.
\end{proof}

The introduction of this new interior transmission problem is also motivated by the following lemma that will play a central role in the differential imaging functional introduced later.
\begin{theorem}
\label{Lem:injectiveG_q2} \label{Lem:injectiveG_q2} Suppose that Assumptions \ref{Ass:nk}, \ref{HypoLSM} and \ref{as-nitp} hold. Then, $\opei^{*}_q \sepmgreq{z} \in \Rang(\opeg^{\pm}_q)$ if and only if $z \in \Dtot_p$.
\end{theorem}
\begin{proof} We first consider the case when  $z \in \Dtot_p = \Dint_p \cup \Dpoutp$ and treat separately the case where $z \in \Dpoutp$ which the part of $\Dtot_p$ that does not intersect the defect and the case where $z$ is the complement part $ \Dpoutp$.

\medskip

{(i) We consider the case $z \in \Dpoutp$:}. Let  $(u, v) \in H^1(D) \times H^1(D)$ be the unique solution of (\ref{oitp}) with $\ftd: =\Phi_q(\cdot-z))|_{\partial D}$ and $\ftn:=\partial   \Phi_q(\cdot-z))/\partial \nu_\ntm|_{\partial D}$  and define 
$$w = \left\{ \begin{array}{cl} u - v
 \quad &\mbox{in} \quad \Dpoutp \\
 \Phi_q \quad &\mbox{in} \quad \bigped \setminus \Dpoutp.
\end{array} \right.$$ Then $w \in H^1_{\loc}(\bigped)$ and verifies equation \eqref{eq:w} with  $f = (f_1,f_2): = (\nabla v, v)$ in $\Dpoutp$ and $f = (- \nabla \Phi_q, -\Phi_q)$ in $\bigped \setminus \Dpoutp$. Therefore $\opeg^{\pm}(f) = \sepmgreq{z}$. Furthermore $u|_{\Dpoutp}$ and $v|_{\Dpoutp}$ are $\alpha_q$-quasi-periodic (due to the periodicity of domain $\Dpoutp$ and $\alpha_q$-quasi-periodicity of the data). This implies $f \in \seth^q_{\inc}(D)$ and $\opeg_q^{\pm}(f) = \opei^{*}_q \opeg_q^{\pm}(f) =  \opei^{*}_q\sepmgreq{z}$. 

\medskip

{ (ii) We consider now the case $z \in \Dint_p$}: We first treat the case $z \in \Dint = \Dint_p \cap \Omega_0$. Let  $(u, v) \in H^1(\Dint_p) \times H^1(\Dint_p)$ be the  $\alpha_q$-quasi-periodic extension of  $(u_{\Dint}, v_{\Dint})$, the solution of  the new interior transmission problem in Definition \ref{nitp} with  $\ftd: =\Phi_q(\cdot;z))|_{\partial \Dint}$ and $\ftn=:\partial   \Phi_q(\cdot;z))/\partial \nu_\ntm|_{\partial \Dint}$. We then define
\[
	w_q = \left\{ \begin{array}{cl} u - v 
 \quad &\mbox{in} \quad \Dint_p \\
 \Phi_q \quad &\mbox{in} \quad \bigped \setminus\Dint_p.
\end{array} \right.
\]
Let $f := (\nabla v, v)$ in $\Dint_p$ and $f := (-\nabla \Phi_q, -\Phi_q)$ in $\bigped \setminus \Dint_p$ then $f \in \seth^q_{\inc}(D)$ and $w_q \in H^1_{\loc}(\bigped)$ satisfies the scattering
problem \eqref{eq:wq} with data $f$. Furthermore, $w$ defined such as $w:=w_q + \opes{k}(f)$ in $\Dint$ and $w: = w_q$ in $D \setminus \Dint$  is solution to \eqref{eq:w} with data $f$. Therefore $\opeg_q^{\pm}(f) =  \opei^{*}_q\sepmgreq{z}$.\\
 We {\it next} consider   $z \in \Dint + mL$ with $0 \neq m \in \ZM$, and  recall that $\sepmgreq{z}= e^{\i mL \cdot \alpha_q} \sepmgreq{z - mL}
$. If we take  $f\in \seth^q_{\inc}(D)$ such that $G_q^\pm(f) =
\opei^{*}_q \sepmgreq{z - mL}$, which is possible by the previous step since
$z - mL \in \Dint$, then 
\[
	 G_q^\pm(e^{\i mL \cdot \alpha_q} f) = \opei^{*}_q(\sepmgreq{z}).
\]

\noindent To conclude the proof we now investigate the case $z \notin \Dtot_p$. If  $G_q^\pm(v) =
\opei^{*}_q \sepmgreq{z}$, then using the same unique continuation
argument as in  the proof of Lemma \ref{Lem:injectiveG_q}  we obtain $w_q =  \Phi_q $ in $\bigped \setminus \Dtot_p$ where $w_q$ is defined by
        \eqref{for:decompw} with $w$ being  the solution of \eqref{eq:w}  with $f=v$. This gives a contradiction since $w_q$ is locally
        $H^1$ in $\bigped \setminus \Dtot_p)$ while $\greq{z}$ is
        not.
\end{proof}

%
%\begin{lemma}
%For $z \in \Dint_p$,  $G_q^\pm(h_z) =
%\opei^{*}_q \sepmgreq{z}$ if and only if $(u_z, v_z)$
%if and only if $ \lim \limits_{\alpha \to 0} (\open_{q,\sharp}^+\tilde \aher^{\alpha,z}_q,\tilde \aher^{\alpha,z}_q) < \infty$. Moreover, if $\zg \in \Dtot_p$ then $\opeh_q^+ \tilde \aher^{\alpha,z}_q \to h_z$ in $L^2(D)$ where $h_z$ is defined by 
%\begin{equation}
%\label{def:fz}
%\left.
%\begin{array}{ll}
%h_z = \left\{ \begin{array}{cl}
%	- \greq{z} \quad  &\text{in}  \quad   \Dint_p \\[1.25ex]
%	v_z  \quad   &\text{in} \quad   \Dpoutp
%\end{array} \right. \qquad  \text{if} \quad   z \in \Dpoutp \\[4.ex]
%%
% h_z = \left\{ \begin{array}{cl}
%	\widehat{v}_z \quad  & \text{in} \quad  \Dint_p \\[1.25ex]
%	- \greq{z} \quad  &\text{in}  \quad   \Dpoutp
%\end{array} \right. \qquad  \text{if} \quad  z \in \Dint_p
%\end{array}
%\right.
%\end{equation}
%
%
% For $z \in \Lambda$
%$G_q^\pm(v_z) =
%\opei^{*}_q \sepmgreq{z}$ if and only if $(u_z, v_z)$ is the solution of problem \eqref{oitp} with $\ftd= \greq{z}$
%and $\ftn = \partial \greq{z}/\partial \nu$ on $\partial D$ and $(\widehat{u}_z,\widehat{v}_z)$ is $\alpha_q$-quasi-periodic extension of  the solution $(u,f)$ of the new interior transmission problem  in Definition \eqref{nitp} with $\ftd= \greq{z}$
%and $\ftn = \partial \greq{z}/\partial \nu$ on $\partial \Dint$.
%\end{lemma}

%This definition is helpful for proving the  following properties for  $\opeh^{\pm}_q$ and $\opeg^{\pm}_q$ that are the counterpart results to  Lemma \ref{lemHerg} and  Lemma \ref{TheoG}, now needed for the operator $\open^{\pm}_q$.

\noindent
\begin{definition}\label{ntp-def}
{\em Values of $k\in {\mathbb C}$ for which the homogenous  problem {\eqref{nitp} with }$\varphi=\psi=0$,  are called {\it new transmission eigenvalues}.
}
\end{definition}

\section{The Analysis of the New Interior Transmission Problem} \label{snitp}
%This section is devoted to the study of the solvability of the new interior transmission problem in Definition \ref{nitp}. It  provides sufficient conditions on  $n$ and $k$ for which this problem is well-posed,  i.e. such that Assumption \ref{as-nitp} holds. As described in the previous section the  solvability of the new interior transmission problem is fundamental to ensuring the properties needed for the imaging of the defect $\omega$ with a single  Floquet-Bloch mode. Up to now the only case that could be handled was when $\omega\cap D=\emptyset$ \cite{tpnguyen} (see also  \cite{Thi-Phong3}). Here we provide a general analysis that cover all possible cases. Our approach generalizes  \cite{kirschitp} and \cite{Sylve2012}. 

\noindent
We are interested in this section by the analysis of the new interior transmission problem as formulated in \eqref{nitp}. We prove that under some reasonable conditions on the material properties and contrasts, this problem is of Fredholm type and the set of new transmission eigenvalues is discrete without finite accumulation point.  We start with proving  the following technical lemma:
\begin{lemma} 
\label{lem:pro1}
There exists $ \theta >0$ and $C > 0$ and $\kappa_0$ independent from $\kappa$ such that 
	\[
		\| \opes{\i\kappa}(f)\|_{H^1(\Dint)} \leq C e^{-\theta \kappa} \| f\|_{H^1(\Dint)}
	\] for all $f \in H^1(\Dint)$ and $\kappa \ge \kappa_0$. 
\end{lemma}

\begin{proof}
Denoting $\twpr: = \opes{i\kappa} (f)$ and $\fc$ the extension of  $f$ as  $\alpha_q-$quasi-periodic in $\Dint_p$, we have that 
\begin{eqnarray}
\label{def:twpr}
	\twpr(x) =& \nabla& \cdot \int_{\Dint_p \setminus \Dint} \fundnp{x - y} \big( (\ntm_p - I) \nabla \fc \, \big)(y) \d{y}   \nonumber\\
	& - &  \kappa^2 \int_{\Dint_p \setminus \Dint} \fundnp{x-y} \big( (\nte_p - 1) \fc \, \big) (y)  \d{y},  \quad  x \in \bigped
\end{eqnarray}
where $\fundnp{\cdot}$ denotes here the $ML-\text{periodic}$ fundamental solution defined in \eqref{phinp} associated with $k=i\kappa$. Let us denote further by 
\begin{equation}
\label{def:wtilde1}
\twpr_1(x) = \nabla  \cdot \int_{\Dint_p \setminus \Dint} \fundnp{x - y} \big( (\ntm_p - I) \nabla \fc \, \big)(y) \d{y} 
\end{equation} and 
\begin{equation}
\twpr_2(x) = -  \kappa^2 \int_{\Dint_p \setminus \Dint} \fundnp{x-y} \big( (\nte_p - 1) \fc \, \big) (y)  \d{y}
\end{equation}
Then $\twpr = \twpr_1 + \twpr_2$. We next define 
\[
	\Sigma := \{x - y, \quad x \in \Dint, \; y \in \Dint_p \setminus \Dint \}, \quad \text{and} \quad d_{max} \in \R: \; d_{max} > \sup \{|z|, z \in \Sigma \}
\] and remark from Assumption \ref{HypoLSM} that $\forall x\in \Dint$, $\forall y\in \Dint_p \setminus \Dint$, $|x - y| >d: = d(\Dint, \Dint_p \setminus \Dint) > 0$. We then have 
\begin{equation}
\label{def:bal}
\Sigma \subset \BB: = B(0, d_{max}) \setminus B(0,d)
\end{equation} where $B(0,d)$ is a ball of radii $d$ and centered at the origin.

\noindent 
An application of the Cauchy-Schwarz inequality, the Fubini theorem and relation \eqref{def:bal} implies
\begin{align*}
	\| \twpr_2\|^2_{L^2(\Dint)} &\leq \kappa^4|\Dint_p\setminus \Dint|\int_{\Dint } \int_{\Dint_p \setminus \Dint} \big|( n_p-1 ) \fc(y) \fundnp{x-y}\big|^2 \d{y} \d{x} \\
	& = \kappa^4 |\Dint_p\setminus \Dint|\int_{\Dint_p \setminus \Dint} \big|( n_p-1 ) \fc(y) \big|^2 \int_{\Dint} \big| \fundnp{x-y} \big|^2\d{x} \d{y}\\
 &\leq \kappa^4|\Dint_p\setminus \Dint| 
	\int_{\Dint_p \setminus \Dint} \big|( n_p-1 ) \fc(y) \big|^2 \d{y}\int_{\BB} \big| \fundnp{z} \big|^2\d{z}.
\end{align*}
Similar we have 
\begin{equation*}
\| \nabla \twpr_2\|^2_{L^2(\Dint)} \leq \kappa^4|\Dint_p\setminus \Dint| 
	\int_{\Dint_p \setminus \Dint} \big|( n_p-1 ) \fc(y) \big|^2 \d{y}\int_{\BB} \big| \nabla  \fundnp{z} \big|^2\d{z}.
\end{equation*}
Since $\fc = f$ in $\Dint$, $\fc$ is quasi-periodic and $n_p$ is periodic in $\Dint_p$, then 
$$
\int_{\Dint_p \setminus \Dint} \big|( n_p-1 ) \fc(y) \big|^2 \d{y} = (|M|-1)\int_{\Dint} \big|( n_p-1 ) \fc(y) \big|^2 \d{y} \leq (|M|-1) \sup_{\Dint}|1 - n_p|\|f \|^2_{L^2(\Dint)}.
$$ Therefore 
\begin{equation}
\|\twpr_2\|^2_{H^1(\Dint)} \leq C \int_{\BB} \Big(\big| \fundnp{z} \big|^2 + \big| \nabla  \fundnp{z} \big|^2\Big) \|f\|^{2}_{L^2(\Dint)},
\end{equation} where \; \; $C: = (|M|-1) \sup_{\Dint}|1 - n_p|$. Following the same line as in the proof of Lemma 4.1 in \cite{Thi-Phong4} using the {fact that } $\ntm_p$ and $\nte_p$  are positive definite we have that
\begin{equation}
\label{proof:tphi}
	\int_{{B}}  \Big(\big| \fundnp{z} \big|^2 + \big| \nabla  \fundnp{z} \big|^2\Big) \d{z} \leq C_0 e^{-\theta \kappa}.
\end{equation} for some constants $C_0 > 0$ and $\theta > 0$. Thus,
\begin{equation}
	\| \twpr_2\|^2_{H^1(\Dint)} \leq C e^{-\theta \kappa} \|f \|^2_{L^2(\Dint)}
\end{equation} with $C = C_0 (|M|-1) \sup_{\Dint}|1 - n_p$. We now estimate $\|\twpr_1\|_{H^1(\Dint)}$ through $\fc$. By the property of convolution, we first write \eqref{def:wtilde1} equivalently as  
\begin{equation}
\twpr_1(x) = \sum_{\ell = 1}^{d} \int_{\Dint_p \setminus \Dint} \Big( \frac{\partial}{\partial x_\ell} \fundnp{x - y}\Big) \big( (\ntm_p - I) \nabla \fc \, \big)(y) \d{y} 
\end{equation}  Using the Cauchy-Schwarz inequality and the Fubini theorem we get again
\begin{align}
\label{rel:w1}
	\|\twpr_1\|^2_{L^2(\Dint)}& \leq& \sum_{\ell = 1}^{d} (M - 1) \|\ntm - I\|_{L^\infty(\Dint)} \|\nabla \fc \|^2_{L^2(\Dint)}  \Big\| \frac{\partial}{\partial x_\ell} \fundnp{x - y}\Big\|^2_{L^2(\BB)}  \nonumber \\
	& = & (M - 1) \|\ntm - I\|_{L^\infty(\Dint)} \|\nabla \fc \|^2_{L^2(\Dint)}  \big\| \nabla  \fundnp{x - y}\big\|^2_{L^2(\BB)}.
\end{align}
We further have that
\begin{equation}
\nabla \twpr_1 = \sum_{\ell = 1}^{d} \int_{\Dint_p \setminus \Dint} \Big( \frac{\partial}{\partial x_\ell} \nabla \fundnp{x - y}\Big) \big( (\ntm_p - I) \nabla \fc \, \big)(y) \d{y} 
\end{equation} This implies using the Cauchy-Schwarz and the Fubini inequalities that
\begin{equation}
\| \twpr_1\|^2_{L^2(\Dint)} \leq (M - 1) \|\ntm - I\|_{L^\infty(\Dint)} \|\nabla \fc \|^2_{L^2(\Dint)}  \Big(\sum_{\ell = 1}^d  \Big\|  \frac{\partial}{\partial x_\ell} \nabla \fundnp{x}\Big\|^2_{L^2(\BB)} \, \Big) 
\end{equation}
From \eqref{rel:w1} and \eqref{rel:graw1} we obtain that
\begin{equation}
\label{rel:graw1}
\|\twpr_1\|^2_{H^1(\Dint)} \leq C \|\nabla \fc \|^2_{L^2(\Dint)} \bigg( \big\| \nabla  \fundnp{\cdot}\big\|^2_{L^2(\BB)} + \sum_{\ell = 1}^d  \Big\|  \frac{\partial}{\partial x_\ell} \nabla \fundnp{x}\Big\|^2_{L^2(\BB)}\bigg),  
\end{equation} with $C: = (M - 1) \|\ntm - I\|_{L^\infty(\Dint)}$. We now prove the exponential decaying of $\big\| \nabla  \fundnp{\cdot}\big\|^2_{L^2(\BB)} + \sum_{\ell = 1}^d  \Big\|  \frac{\partial}{\partial x_\ell} \nabla \fundnp{x}\Big\|^2_{L^2(\BB)}$.  However, by \eqref{proof:tphi} we already have the exponential decaying of $\big\| \nabla  \fundnp{\cdot}\big\|^2_{L^2(\BB)}$. So it leads to estimate that 
\begin{equation}
	\sum_{\ell = 1}^d  \Big\|  \frac{\partial}{\partial x_\ell} \nabla \fundnp{x}\Big\|^2_{L^2(\BB)} \leq Ce^{- \theta \kappa}
\end{equation} for some constants $C > 0$ and $\kappa > 0$. Recall that $\fundnp{x}$ satisfies 
\begin{equation}
\label{recall:phin}
\nabla \cdot \ntm_p \nabla \fundnp{x} - \kappa^2 \nte_p \fundnp{x} = 0 \quad \text{in} \quad \Dint_p \setminus \Dint.
\end{equation}
Taking the partial derivative of equation \eqref{recall:phin} with respect to $x_\ell$ for all $ \ell = 1, \ldots, d$, we obtain
\begin{equation}
\label{atmood1}
	\nabla \cdot \frac{\partial}{\partial x_\ell} \Big(\ntm_p \nabla \fundnp{x} \Big) - \kappa^2 \frac{\partial}{\partial x_\ell} \Big( \nte_p \fundnp{x}\Big) = 0
\end{equation} We denote by $\widehat{\ntm}^\ell_p: =\frac{\partial}{\partial x_\ell} \ntm_p$ and  $\widehat{\Phi}^\ell(n_p,x): = \frac{\partial}{\partial x_\ell} \fundnp{x}$. From \eqref{atmood1} we have 
{
\begin{equation}
\label{atmood2}
	\nabla \cdot \ntm_p \nabla \widehat{\Phi}^\ell(n_p; x) - \kappa^2 n_p \widehat{\Phi}^\ell(n_p; x) =  \nabla \cdot \widehat{\ntm}^\ell_p \nabla \fundnp{x} - \kappa^2 n_p \widehat{\Phi}^\ell(n_p; x)   +\kappa^2 \frac{\partial}{\partial x_\ell} \Big( \nte_p \fundnp{x}\Big) . 
\end{equation} 
We observe that the $H^{-1}(\tilde B)$ norm of the right hand side is exponentially small with respect to $\kappa$ for any bounded domain not containing the origin. Therefore, as in the proof of the exponential decay for  $\fundnp(\cdot)$,  multiplying \eqref{atmood2} with $\chi \widehat{\Phi}^\ell(n_p, \cdot)$ with $\chi$ a $C^\infty$ cutoff function that vanishes in a neighborhood of the origin and is $1$ in $B$, one can prove that  
%\begin{multline}
%	\int_{\Omega} \Big(\ntm_p \nabla \widehat{\Phi}^\ell(n_p,x) \cdot \nabla \widehat{\Phi}^\ell(n_p,x) + \kappa^2 \nte_p |\widehat{\Phi}^\ell(n_p,x)|^2\Big) \d{x} = \\
%	\int_{\bigped} \Big(\widehat{\ntm}^\ell_p \nabla \fundnp{x} \cdot \nabla \widehat{\Phi}^\ell(n_p,x) + \kappa^2 \widehat{\nte}^\ell_p \fundnp{x} \cdot \widehat{\Phi}^\ell(n_p,x) \Big) \d{x} 
%\end{multline} Since $\ntm_p$ and $\nte_p$ are positive defined, there exists a constant $c_0 > 0$ and $c_1 >0$ such that 
\begin{equation}
	\|\widehat{\Phi}^\ell(n_p,\cdot)\|_{H^1(B)} \leq C e^{-\theta \kappa}
\end{equation}
for some possibly different positive constants $C$ and $\theta$ but which are independent for $\kappa$. }This ensure (from \eqref{rel:graw1}) that, there exists a constant $\tilde{C} >0$ such that
\begin{equation}
	\|\twpr_1\|_{H^1(\bigped)} \leq C e^{-\theta \kappa} \| \nabla f\|_{L^2(\bigped)}
\end{equation} which end of the proof.

\end{proof}

\noindent
We now turn our attention to the analysis of the new interior transmission problem in Definition \ref{nitp}. To further simplify notation, we set $\lambda: = - k^2\in{\mathbb C}$, ${ F_1(f)}: =  \big (\ntm_p - \ntm) \nabla \opes{\sqrt{-\lambda}}(f)$ and ${ F_2(f)} :=   (n_p - n) \opes{\sqrt{-\lambda}}(f)$. With these notations, the problem we need to solve reads: Find $(u, f) \in H^1(\Dint) \times H^1(\Dint)$ such that
\begin{equation} 
\label{anisonitp}
\left\{ \begin{array}{llll}
\nabla \cdot \ntm \nabla  {\eg} - \lambda n {\eg}  - \nabla \cdot { F_1(f)} + \lambda { F_2(f)}&=&  0 \quad& \mbox{ in } \; \Dint, 
\\[6pt]
\Delta f - \lambda f &=& 0  \quad &\mbox{ in } \; \Dint,
\\[6pt]
 {\eg} - f&=& {\ftd} \quad &\mbox{ on } \; \partial \Dint,
\\[6pt]
\partial{\eg}/\partial \nu_\ntm - {F_1(f) \cdot \nu} - \partial f/\partial \nu &=& {\ftn} \quad &\mbox{ on } \; \partial \Dint, 
\end{array}\right.
\end{equation} for given $({\ftd},\ftn) \in H^{1/2}(\partial \Dint) \times H^{-1/2}(\partial \Dint) $.  {Let us consider the Hilbert space }
\begin{equation}
	\Csp: = \{(\varphi, \psi) \in H^1(\Dint) \times H^1(\Dint) \; \text{such that} \; \varphi = \psi \; \text{on} \; \partial \Dint\}.
\end{equation}
For a given $\ftd \in H^{1/2}(\Dint)$ we first construct a lifting function $u_0 \in H^1(\Dint)$ such that $u_0 = \ftd$.
 We then write  the interior transmission problem \eqref{anisonitp} equivalently in a variational form as follows: {find  $(u-u_0, f)\in \Csp$ }such that 
\begin{align}
\label{var-anis}
&\int\limits_{\Dint} \ntm \nabla u \cdot \nabla\overline{\varphi} \,\d{x} - \int\limits_\Dint \nabla f \cdot \nabla\overline{\psi}\,\d{x }  - \int\limits_{\Dint} F_1(f) \cdot \nabla\ol{\varphi} 
+ \lambda \int\limits_D \nte  u\,\overline{\varphi}\,\d{x} - \lambda\int\limits_D f\, \overline{\psi}\,\d{x}\nonumber\\
& -  \lambda \int\limits_{\Dint} F_2(f) \ol{\varphi} \, =\, \int_{\partial \Dint} h \overline{\varphi}\,ds \quad \mbox{for all} \quad (\varphi, \psi) \in \Csp.
\end{align}
 Let us define the bounded sesquilinear forms $a_\lambda(\cdot,\cdot)$ by
\begin{multline}
a_\lambda((u,f), (\varphi,\psi)):=\int\limits_\Dint \ntm \nabla u \cdot \nabla\overline{\varphi}\,\d{x} - \int\limits_\Dint \nabla f \cdot \nabla\overline{\psi}\,\d{x}  - \int\limits_\Dint F_1(f) \cdot \nabla \ol{\varphi}\\
+ \lambda \int\limits_\Dint \nte  u\,\overline{\varphi}\,\d{x}  - \lambda \int\limits_\Dint f\, \overline{\psi} \,\d{x} -  \lambda\int\limits_\Dint F_2(f) \ol{\varphi} \d{x}
\end{multline}
and the bounded antilinear functional $L: \Csp \to \CC$ by
$${L (\varphi,\psi):=\int_{\partial \Dint} h \overline{\psi}\,ds- a_\lambda((u_0,0), (\varphi,\psi)).}$$
Letting ${\mathbf A}: \Csp \to \Csp$ be the bounded linear  operator defined by means of the Riesz representation theorem
\begin{equation}
\label{def:Ak}
 \left({\mathbf A}_\lambda(v,f), (\varphi, \psi)\right)_{\Csp} = a_\lambda((v,f), (\varphi,\psi))
 \end{equation} and $\ell \in \Csp$ the Riesz representative of $L$ defined by
$$\left(\ell, (\varphi,\psi)\right)_{\Csp} =L(\varphi,\psi),$$
the interior transmission problem becomes find $(u-u_0,f)\in \Csp$ satisfying
$${\mathbf A}_\lambda(u-u_0,f) =\ell. $$ Hence if is sufficient to prove that ${\mathbf A}_{\kappa}$ is invertible for some $\kappa > 0$ and ${\mathbf A}_\lambda - {\mathbf A}_{\kappa}$ is compact in order to conclude that ${\mathbf A}_\lambda$ is a Fredholm operator of index zero. {Analytic Fredholm theory then implies that the set of new transmission eigenvalues is discrete without finite accumulation points}. We assume   that there exists a $\delta$-neighborhood ${\mathcal N}$ of the boundary $\partial \Dint$ in $\Dint$ i.e. 
$${\mathcal N}:=\left\{x\in \Dint: \mbox{ dist}(x,\partial \Dint)<\delta\right\}$$  such that $\Im(\ntm)=0$ and  $\Im(\nte)=0$   in ${\mathcal N}$ and either $0<a_0<a^\star<1$, $0 < n_0 < n^\star <1 $  or $a_\star>1$, $n_\star > 1$  where
\begin{equation}\label{bornes constantes-2}
\begin{array}{lll}
a_\star:=\underset{x\in {\mathcal N}}{\mbox{ inf }}\underset{\tiny \begin{array}{cll}\xi\in\mathbb{R}^3 \\ |\xi|=1\end{array}}{\mbox{ inf }} \xi\cdot \ntm(x)\xi>0,\quad n_\star:  = \underset{x\in {\mathcal N}}{\mbox{ inf }} \nte(x) > 0&  \\
&\\
a^{\star}:=\underset{x\in {\mathcal N}}{\mbox{ sup }}\underset{\tiny \begin{array}{cll}\xi\in\mathbb{R}^3 \\ |\xi|=1\end{array}}{\mbox{ sup }}\xi\cdot \ntm(x)\xi<\infty, \quad n^\star:  = \underset{x\in {\mathcal N}}{\sup} \,\nte(x) <\infty. 
\end{array}
\end{equation}
Let us start with the case when $a_0<a^*<1$. For {later} use, we introduce $\chi\in\mathcal{C}^{\infty}(\overline{\Dint})$ a cut off function such that $0\leq\chi\leq 1$ is supported in $\overline{\mathcal N}$ and equals to one in a neighborhood of the boundary. 
\begin{lemma}
\label{LemmeInver}
Assume that $\ntm$ and $n$ are real valued in ${\mathcal N}$ and either $0<a_0<a^\star<1$, $0 < n_0 < n^\star <1 $  or $a_\star>1$, $n_\star > 1$. Then, for sufficient large $\kappa > 0$, the operator $\mathbf{A}_{\kappa}$ is invertible.
\end{lemma}
\begin{proof} 
We shall prove first the case $0<a_0<a^\star<1$, $0 < n_0 < n^\star <1 $. Using the $T-$coercivity approach \cite{Bonne2011a}, we first define the isomorphism ${\mathbf T}: \Csp\to \Csp$ by
$${\mathbf T}: (u,f) \mapsto (u-2\chi f,-f)$$  (Note that ${\mathbf T}$ is an isomorphism since ${\mathbf T}^2=I$). {We then consider the sesquilinear form $a_\lambda^{\mathbf{T}}$ defined on $\Csp \times \Csp$ by
$$
a_\lambda^{\mathbf{T}}((u,f),(\varphi,\psi)) = a_\lambda^{\mathbf{T}}((u,f),{\mathbf T}(\varphi,\psi)).
$$
To prove the lemma, it is sufficient to  prove that $a_{\kappa}^{\mathbf{T}}$ is coercive for $\kappa$ sufficiently large.}
We have for all $(u,f)\in \Csp$, 
\begin{align}
 a_{\kappa}^{\mathbf{T}}((u,f),(u,f))  &=  \int\limits_\Dint \ntm \nabla u \cdot \nabla u + |\nabla f|^2 - 2 \ntm\nabla u \cdot\nabla(\chi f)) - F_1(f) \cdot \nabla (u - 2\chi f) \d{x}\nonumber \\
 & \qquad + \kappa  \int\limits_{\Dint} \nte |u|^2 +|f|^2 - 2 \nte u \,\ol{\chi f}  - F_2(f) \,\ol{(u - 2 \chi f)} \d{x}.  \label{Tcoer etape 1}
\end{align}
From Lemma \ref{lem:pro1} and the inequality $(ax + by)^2 \leq (a^2 + b^2)(x^2 + y^2)$ we have 
\begin{align}
\label{eq:tg1}
	&\Big|\int_{\Dint} F_1(f) \cdot \nabla (u -2\chi f)  \Big|  + \kappa \Big|\int_{\Dint} F_2(f)  \ol{u -2\chi f} \Big| \\
	&= \Big| \int_{\Dint} (\ntm_p - \ntm) \nabla \opes{i \sqrt{\kappa}}(f) \cdot \nabla \ol{u -2\chi f}\Big| + \Big| \kappa^2\int_{\Dint}(n_p - n)  \opes{i \sqrt{\kappa}}(f) \ol{u -2\chi f} \Big|  \nonumber\\
	&\leq \max\{\|\ntm_p - \ntm \|_{L^{\infty}(\Dint)} , \kappa^2 \|\nte_p - \nte \|_{L^{\infty}\Dint)}  \} Ce^{-\theta \sqrt{\kappa}}  \| f\|_{H^1(\Dint)}  \| u -2\chi f \|_{H^1(\Dint)} 
\end{align} 
where the quantity $Ce^{-\theta \sqrt{\kappa}}$ is defined in Lemma \ref{lem:pro1}.  By Cauchy-Schwarz inequality we have the following estimate
\begin{align*}
 \| f\|_{H^1(\Dint)}  \| u -2\chi f \|_{H^1(\Dint)}  &\leq  \|f\|^2_{H^1(\Dint)} + \frac{1}{4}  \|u -2\chi f\|^2_{H^1(\Dint)} \\
 &\leq  \|f\|^2_{H^1(\Dint)} + \Big( \|u\|^2_{H^1(\Dint)} +4 \max\{1, \|\nabla \chi \|_{L^\infty(\mathcal{N})} \} \big) \|f\|^2_{H^1(\mathcal{N})} \Big)
\end{align*} Let us denote by $ c_0(\kappa): =  \max\{\|\ntm_p - \ntm \|_{L^{\infty}(\Dint)} , \kappa^2 \|\nte_p - \nte \|_{L^{\infty}\Dint)}  \} Ce^{-\theta \kappa}$ and $c_1(\kappa) := 4 \max\{1, \|\nabla \chi \|_{L^\infty(\mathcal{N})} \} c_0(\kappa)$ we then have 
\begin{multline}
\label{eq:tglem1}
\Big|\int_{\Dint} F_1(f) \cdot \nabla (u -2\chi f)  \Big|  + \kappa \Big|\int_{\Dint} F_2(f)  \ol{u -2\chi f} \Big| \\  
\leq ({c_0(\kappa) + c_1(\kappa))} \|f\|^2_{H^1(\Dint)}  + c_0(\kappa) \|u\|^2_{H^1(\Dint)}
\end{multline}
Furthermore, using Young's inequality,  we can write 
\begin{align}
 2\left| \int_{\Dint} \ntm\nabla u\cdot\nabla(\chi f))\right| & \le \;  2\left|\int_{\mathcal{N}} \chi \ntm\nabla u\cdot\nabla f\right|  + 2 \left|\int_{\mathcal{N}} \ntm \nabla u \cdot \nabla(\chi) f\right| \nonumber   
 \\ 
 &  \le \;  \alpha\int_{\mathcal{N}}\left|\ntm  \nabla u \cdot \nabla u \right| + \alpha^{-1} \int_{\mathcal{N}} \left|\ntm \nabla f\cdot\nabla f\right| \label{inegalite Young 1}\\
 &  + \; \beta \int_{\mathcal{N}} \left|\ntm\nabla u \cdot \nabla u \right|+ \beta^{-1}\int_{\mathcal{N}}\left|\ntm \nabla(\chi)\cdot\nabla(\chi)\right||f|^2 \nonumber 
\end{align} 
and 
\begin{equation}\label{inegalite Young 11}
2\left| \int_{\Dint} \nte u \,\chi f\right| \le  \eta \int_{\mathcal{N}} \nte |u|^2 + \eta^{-1} \int_{\mathcal{N}} \nte |f|^2
\end{equation}
for arbitrary constants $\alpha>0$, $\beta>0$ and $\gamma>0$.  
Substituting \eqref{eq:tglem1}, (\ref{inegalite Young 1}) and (\ref{inegalite Young 11}) into (\ref{Tcoer etape 1}), we  now obtain
\begin{align}
\left|a_{\kappa}^{\mathbf{T}}((u,f),(u,f))\right| & \ge  \int\limits_{\Dint \setminus \mathcal{N}} \Re(\ntm)  \nabla u \cdot\nabla \bar u + \int\limits_{\Dint \setminus \mathcal{N}}  |\nabla f |^2 + \kappa \int\limits_{\Dint \setminus \mathcal{N}}  \Re (\nte)  | u|^2 + \kappa \int\limits_{\Dint \setminus \mathcal{N}} |f|^2 \nonumber\\
 & \; +   \int\limits_{\mathcal{N}}\big((1- \alpha - \beta\big) \ntm \nabla u \cdot \nabla \bar u +
\big((I- \alpha^{-1}\ntm\big)\nabla f \cdot \nabla \bar f\nonumber \\
 & + \kappa \int\limits_{\mathcal{N}} (1- \eta)\nte |u|^2 + \int\limits_{\mathcal{N}}\big((\kappa (1- \eta^{-1}\nte)-  \|\nabla\chi\|^2_{L^\infty(\mathcal{N})}a^{\star}\alpha^{-1}\big)|f|^2\nonumber\\
 &-({c_0(\kappa) + c_1(\kappa))} \|f\|^2_{H^1(\Dint)}  - c_0(\kappa) \|u\|^2_{H^1(\Dint)}. \nonumber
\end{align}
Taking $\alpha$, $\beta$, and $\eta$ such that $a_0 < \alpha < 1$, $n_0 < \eta <1$ and $\beta + \alpha < 1$ we then get 
\begin{align}
\left|a_{\kappa}^{\mathbf{T}}((u,f),(u,f))\right| & \ge \gamma_1 \|\nabla u \|^2_{L^2(\Dint)} + \kappa \gamma_2 \|u \|^2_{L^2(\Dint)}   +  \gamma_3  \|\nabla f \|^2_{L^2(\Dint)}  + (\gamma_4 \kappa - \gamma_5) \|f\|^2_{L^2(\Dint)}  \nonumber\\
 &-({c_0(\kappa) + c_1(\kappa))} \|f\|^2_{H^1(\Dint)}  - c_0(\kappa) \|u\|^2_{H^1(\Dint)} \nonumber
\end{align}
for some constants $\gamma_i$, $i=1, \ldots, 5$ that are positive and independent from $\kappa$. Since $c_0(\kappa) $ and $c_1(\kappa) $ go to $0$ as $\kappa \to \infty$ one then easily obtains the coercivity of $a_{\kappa}^{\mathbf{T}}$ for large enough $\kappa$.  This finishes the proof of the case $0<a_0<a^\star<1$, $0 < n_0 < n^\star <1 $.  The proof of the case  $a_\star>1$, $n_\star > 1$ follows the same lines using the isomorphism ${\mathbf T}: (u,f)\mapsto (u,  2\chi u -f)$.
\end{proof}

\begin{lemma} For any complex numbers $\lambda$ and $\kappa$, the operator  
$A_\lambda - A_{\kappa} : \Csp \to \Csp$ is compact.
\end{lemma}

\begin{proof}
	Taking the difference  $a_\lambda - a_{\kappa}$ we have
\begin{align*}
&a_\lambda((u,f), (\varphi, \psi)) - a_{\kappa}((u,f), (\varphi, \psi)) =  \\
&(\kappa - \lambda) \int\limits_{\Dint} F_1(f) \cdot \nabla\ol{\varphi} 
+(\kappa - \lambda) \int\limits_{\Dint} F_2(f) \ol{\varphi}  + (\lambda - \kappa) \int\limits_\Dint \nte  u\,\overline{\varphi}\,\d{x} - (\lambda -\kappa)\int\limits_\Dint f\, \overline{\psi}\,\d{x}.
\end{align*}
The compactness of $A_\lambda - A_{\kappa}$ then easily follows from the continuity of $F_1 : L^2(\Lambda) \to \Csp $  and $F_2 :  L^2(\Lambda) \to \Csp$ and the compact embedding of $H^1(\Lambda)$  into $L^2(\Lambda)$.
\end{proof}

As a consequence of the two previous lemma and analytic Fredholm theory we get the following result on new transmission eigenvalues. 
Note that this theorem provides sufficient conditions under which Assumption \ref{as-nitp} hold.
 \begin{theorem}\label{mainth}
Assume that the hypothesis of Lemma \ref{LemmeInver} hold.  Then the new interior transmission formulated in Definition \ref{nitp} has a unique solution depending continuously on the data $\varphi$ and $\psi$ provided $k\in{\mathbb C}$ is not a new transmission eigenvalue defined  in Definition \ref{ntp-def}. In particular the set of new transmission eigenvalues in ${\mathbb C}$ is discrete (possibly empty) with $+\infty$  as the only possible accumulation point.
\end{theorem}

%%%%%%%%%%%%%%%%%%%%%%%%%%%%%%%%%%%%%%%%%%%%%%%%%%%%%%%%%%%%%%
\section{A Differential Imaging Algorithm} \label{ch3subsec4}
%We now apply all the results of Section 2 and Section 4 above to design an algorithm that provides us with the support of the defect $\domp$ without reconstructing $D_p$ or computing the Green's function of the periodic media.   {We follow the idea proposed in  \cite{Thi-Phong3} and build a differential imaging functional by comparing the application of the Generalized Linear Sampling algorithm to respectively the operators $N^\pm$ and  $N_q^{\pm}$}. The new results obtained in Theorem \ref{mainth} allow us to justify this algorithm for general location of $\omega$ (possibly multi-component).
\subsection{{Description and analysis of the algorithm} }
Throughout this section we assume that  Assumptions \ref{Ass:nk}, \ref{HypoLSM} and \ref{as-nitp}  hold. For sake of simplicity of presentation  we
only state the results when  the measurements operator $\open^+$ is available. Exactly the same
holds for the operator  $\open^-$ by changing everywhere the exponent  $+$  to $-$. For given $\phi$ and $a$ in $\ell^2(\Zd)$ we define the functionals
\newcommand{\pmp}{+}
\begin{equation}
	\begin{array}{ll}
		J^{\pmp}_{\alpha}(\phi,\aher) := \alpha(\open_{\sharp}^\pmp\aher,\aher) + \|\open^\pmp\aher - \phi\|^2, \\[1.5ex]
		J^{\pmp}_{\alpha,q}(\phi,\aher) :=
                \alpha(\open_{q,\sharp}^\pmp\aher,\aher) + \|\open_q^\pmp \aher - \phi\|^2
	\end{array}
\end{equation}
{with $\open_{q,\sharp}^\pmp:= I_q^* \open_{\sharp}^\pmp  I_q$.}
Let $\aher^{\alpha,z}$, $\aher_q^{\alpha,z}$ and $\tilde \aher^{\alpha,z}_q$ 
in $\ell^2(\Zd)$ verify (i.e. are minimizing sequences)
\begin{equation}
	\begin{array}{ll}
\displaystyle 		J^{\pmp}_{\alpha} (\sepgre{z},\aher^{\alpha,z}) \leq \inf_{\aher \in \ell^2(\Zd)} J^{\pmp}_{\alpha}(\sepgre{z},a) + c(\alpha)
		\\[12pt]
\displaystyle 		J^{\pmp}_{\alpha} (\sepgreq{z},\aher_q^{\alpha,z}) \leq \inf_{\aher \in \ell^2(\Zd)} J^{\pmp}_{\alpha}(\sepgreq{z},a) + c(\alpha)
		\\[12pt]
\displaystyle 			J^{\pmp}_{\alpha,q} (\opei^{*}_q \sepgreq{z},\tilde \aher_q^{\alpha,z}) \leq \inf_{\aher \in \ell^2(\Zd)} J^{\pmp}_{\alpha,q}(\opei^{*}_q \sepgreq{z},a) + c(\alpha)
	\end{array}
\end{equation}
with $\frac{c(\alpha)}{\alpha} \to 0$ as $\alpha \to 0$.  Here $\sepmgre{z}$  are  the Rayleigh coefficients of $\Phi(x;z)$ given by \eqref{lunch} and $\sepmgreq{z}$  are the Rayleigh coefficients of $\greq{z}$ given by (\ref{hhh}).

Based on the results of the previous sections and following the same arguments as in \cite[Section 6]{Thi-Phong4} we obtain the following result that we state here without proof.
\begin{lemma}\label{glsm}
\begin{enumerate}
\item  $\zg \in D$ if and only if $ \lim \limits_{\alpha \to 0} (\open_{\sharp}^+\aher^{\alpha,z},\aher^{\alpha,z})< \infty$. Moreover, if $\zg \in D$ then $\opeh^+ \aher^{\alpha,z} \to v_z$ in $L^2(D)$ where $(u_z, v_z)$ is the solution of problem \eqref{oitp} with $\ftd= \Phi(x;z)$
and $\ftn = \partial \Phi(x;z)/\partial \nu$ on $\partial D$.
\item  $\zg \in D_p$ if and only if $ \lim \limits_{\alpha \to 0} (\open_{\sharp}^+\aher_q^{\alpha,z},\aher_q^{\alpha,z})< \infty$. Moreover, if $\zg \in D_p$ then $\opeh^+ \aher_q^{\alpha,z} \to v_z$ in $L^2(D)$ where $(u_z, v_z)$ is the solution of problem \eqref{oitp} with $\ftd= \greq{z}$
and $\ftn = \partial \greq{z}/\partial \nu$ on $\partial D$.
\item $\zg \in \Dtot_p$ if and only if $ \lim \limits_{\alpha \to 0} (\open_{q,\sharp}^+\tilde \aher^{\alpha,z}_q,\tilde \aher^{\alpha,z}_q) < \infty$. Moreover, if $\zg \in \Dtot_p$ then $\opeh_q^+ \tilde \aher^{\alpha,z}_q \to h_z$ in $L^2(D)$ where $h_z$ is defined by 
\begin{equation}
\label{def:fz}
\left.
\begin{array}{ll}
h_z = \left\{ \begin{array}{cl}
	- \greq{z} \quad  &\text{in}  \quad   \Dint_p \\[1.25ex]
	v_z  \quad   &\text{in} \quad   \Dpoutp
\end{array} \right. \qquad  \text{if} \quad   z \in \Dpoutp \\[4.ex]
 h_z = \left\{ \begin{array}{cl}
	\widehat{v}_z \quad  & \text{in} \quad  \Dint_p \\[1.25ex]
	- \greq{z} \quad  &\text{in}  \quad   \Dpoutp
\end{array} \right. \qquad  \text{if} \quad  z \in \Dint_p
\end{array}
\right.
\end{equation}
 where $(u_z, v_z)$ is the solution of problem \eqref{oitp} with $\ftd= \greq{z}$
and $\ftn = \partial \greq{z}/\partial \nu$ on $\partial D$ and $(\widehat{u}_z,\widehat{v}_z)$ is $\alpha_q$-quasi-periodic extension of  the solution $(u,f)$ of the new interior transmission problem  in Definition \eqref{nitp} with $\ftd= \greq{z}$
and $\ftn = \partial \greq{z}/\partial \nu$ on $\partial \Dint$.
\end{enumerate}
\end{lemma}

\noindent
We then consider the following imaging functional that characterizes $\Dint$,
\begin{equation} \label{IndFuncDiffIm}
\I^{\pmp}_\alpha(z) = \left({(\open_{\sharp}^\pmp a^{\alpha,z}, a^{\alpha,z}) \left(1 + \frac{(\open_{\sharp}^\pmp a^{\alpha,z}, a^{\alpha,z})}{D^{\pmp}(a^{\alpha,z}_q, \tilde a^{\alpha,z}_q)}\right)}\right)^{-1}
\end{equation}
where for $a$ and $b$ in $\ell^2(\Zd)$,  
$$ D^{\pmp}(a, b):= \left(\open_{\sharp}^\pmp (a - \opei_q b), (a - \opei_q b)\right). $$ 

\noindent Based on Lemma \ref{glsm}, we can show in the following Theorem that the  functional  $\I^{\pmp}_{\alpha}(z)$  provides an indicator function for $D\setminus \Dpout$, i.e. the defect and the periodic components of the background that intersects $\omega$. 

\begin{theorem} \label{DiffImagFunc}
Under Assumptions \ref{Ass:nk}, \ref{HypoLSM}, \ref{as-nitp} and the Assumption that the following interior transmission problem has only trivail solution 
\begin{equation}
\left\{
	\begin{array}{ll}
	\nabla \cdot \ntm \nabla u + k^2 \nte u = 0 &\quad \text{in} \; \; \domp \\[1.25ex]
	\nabla \cdot \ntm_p \nabla v + k^2 \nte_p v = 0 &\quad \text{in} \; \; \domp \\[1.25ex]
	u - v = 0 &\quad \text{on} \; \; \partial \domp \\[1.25ex]
	\nu \cdot \ntm \nabla u - \nu \cdot \ntm_p \nabla v = 0 &\quad \text{on} \; \; \partial \domp
	\end{array}
\right.
\end{equation} we have 
$$z \in D\setminus \Dpoutp \mbox{\;\;  if and only if \;\; } \lim_{ \alpha \to 0} \I^{\pmp}_\alpha(z) > 0.$$ 
(Note that $D\setminus \Dpoutp = \domp \cup \Dpint_p$ contains the physical defect $\omega$ and {$\Dpint_p := D_p \setminus \Dpoutp$} the components of $D_p$ which have nonempty intersection with the defect).
\end{theorem}
\begin{proof}
The proof of Theorem \ref{DiffImagFunc} follows the same line as Theorem 5.2 in \cite{Thi-Phong4}.
\end{proof}

\noindent
We recall that exactly the same can be shown for down-to-up incident field, by simply replacing the upper index $+$ with $-$.  It is also possible to handle the case with noisy data, and we refer the reader to  \cite{Thi-Phong3} and  \cite{tpnguyen} for more detailed discussion.

\subsection{Numerical Experiments}
\label{sec:numerical}
%\mfied{Il faut mettre les valeurs de A dans les exemples numeriques !!!! et refaire l'exemple 1}
We conclude by showing several numerical examples to test  our differential imaging algorithm. We limit ourselves to examples in ${\mathbb R}^2$. The data  is computed   with both  down-to-up and  up-to-down plane-waves  by solving the forward scattering problem based on the spectral discretization  scheme of  the volume integral formulation of the problem presented in \cite{Thi-Phong2}.

\noindent
Let us  denote by 
$$Z^{d-1}_{inc}:=\{j = q + M\ell, \;  q \in \ZM, \; \ell \in \Zd \; \text{and} \; \ell \in \interd{-N_{min}}{N_{max}} \}$$
the set of indices for the incident waves (which is also the set of indices for measured Rayleigh coefficients). The values of all  parameters used in our experiments  will be indicated below. The discrete version of the operators $\open^{\pm}$ are given by the $N_{inc} \times N_{inc}$ matrixes
\begin{equation}
\label{discre:MatrixN}
\open^{\pm} := \left(\raycoefpm{u^s}{\ell;j}\right)_{\ell,j \in \Zd_{inc}}.	
\end{equation}
%where the index $p$ to indicate the correspondence with respect to periodic domain (which has refractive index $n_p$). 
%%The coefficients of the matrices $\open^{\pm}$ and $\open^{\pm}_p$ can be calculated from the Fourier coefficients as:
%%\begin{equation}
%%% 
%%\end{equation}
Random noise is added to the data. More specifically, we in our computations we use 
\begin{equation}
\label{discre:MatrixNnoise}
\open^{\pm,\delta}(j,\ell): = \open^{\pm}(j,\ell)\big(1 + \delta A(j,\ell)\big), \quad \forall (j,\ell) \in \Zd_{inc}\times \Zd_{inc}
% \quad \mbox{and} \quad \open_p^{\pm,^\delta}: =  \open_p^{\pm}(Id + \delta N)
\end{equation}
where $A = (A(j,\ell))_{N_{inc}\times N_{inc}}$ is a matrix of uniform complex
random variables with real and imaginary parts in $[-1,1]^2$ and  $\delta > 0$
is the noise level. In our examples we take $\delta = 1\%$.

\noindent
For  noisy data, one needs to redefine the functionals $J^{\pmp}_{\alpha}$
and $J^{\pmp}_{\alpha, q}$ as \begin{equation}
	\begin{array}{ll}
		J^{\pmp,\delta}_{\alpha}(\phi,\aher) := \alpha\left(
                  (\open_{\sharp}^{\pmp,\delta}\aher,\aher) + \delta
                  \|\open_{\sharp}^{\pmp,\delta}\|
                  \|a\|^2\right) +  \|\open^{\pmp, \delta}\aher - \phi\|^2, \\[1.5ex]
		J^{\pmp,\delta}_{\alpha,q}(\phi,\aher) :=\alpha\left(
                  (\open_{\sharp}^{\pmp,\delta}\opei_q \aher,\opei_q\aher) + \delta
                  \|\open_{\sharp}^{\pmp,\delta}\|
                  \|a\|^2\right) +  \|\open^{\pmp, \delta}_q\aher - \phi\|^2
	\end{array}
\end{equation}
 We then consider  $\aher^{\alpha,z}_\delta$, $\aher_{q,\delta}^{\alpha,z}$ and $\tilde \aher^{\alpha,z}_{q,\delta}$ 
in $\ell(\Zd)$ as the minimizing sequence  of, respectively,  
$$
J^{\pmp,\delta}_{\alpha}(\sepgre{z},\aher), \;
J^{\pmp,\delta}_{\alpha}(\sepgreq{z},\aher) \mbox{ and } J^{\pmp,\delta}_{\alpha,q}(\sepgreq{z},\aher).
$$
The noisy indicator function takes the form
\begin{equation} \label{IndFuncDiffImNoisy}
\I^{\pmp,\delta}_\alpha(z) = \left({\mathcal{G}^{+,\delta}(a^{\alpha,z}_\delta)\left(1 +
    \frac{\mathcal{G}^{+,\delta}(a^{\alpha,z}_\delta)}{D^{\pmp, \delta}(a^{\alpha,z}_{q,\delta}, \tilde a^{\alpha,z}_{q,\delta})}\right)}\right)^{-1}
\end{equation}
where for $a$ and $b$ in $\ell^2(\Zd)$,  
$$ D^{\pmp,\delta}(a, b):= \left(\open_{\sharp}^{\pmp,\delta} (a - \opei_q b),
  (a - \opei_q b)\right) $$
and 
$$
\mathcal{G}^{+,\delta}(\aher) :=  (\open_{\sharp}^{\pmp,\delta}\aher,\aher) + \delta
                  \|\open_{\sharp}^{\pmp,\delta}\|
                  \|a\|^2.
$$

\noindent
Defining in a similar way the  indicator function
$\I^{-,\delta}(z)$ corresponding to up-to-down incident waves, we use the following indicator function in our numerical examples
$$
\I^\delta(z): = \I^{+,\delta}(z) + \I^{-,\delta}(z).
$$
In the three first examples, we consider the periodic background $D_p$, in which each cell consists  of two circular components, namely the  discs with radii $r_1$, $r_2$ specified below. The physical parameters are set as 
\begin{equation}
\label{Num:parameters}
k = 3.5\pi/3.14; \; \; A_p = I, n_p = 2 \text{ inside the discs}, \text{and }  A_p = I , n_p = 1\;  \text{otherwise}. 
\end{equation}
Letting $\lambda := 2\pi/k$ be the wavelength,  the geometrical parameters are set as
\begin{equation}
\label{Num:parameters1}
\text{the period length }  \, L = \pi \lambda, \; \text{half width of the layer} \, h  =  1.5 \lambda, \; r_1 = 0.3 \lambda,  \; \text{and} \; r_2 = 0.4 \lambda.
\end{equation}
Finally we choose the truncated model 
\begin{equation}
\label{Num:parameters2}
		 M = 3, \;\;  N_{min} = 5\;  \mbox{ and } \; N_{max} = 5 { \mbox{ and } q = 1}
\end{equation} 
The reconstructions are displayed by plotting  the indicator function $\I^\delta(z)$. 

\subsubsection*{Example 1.}
In the first example, we consider the perturbation $\domp$ to be a disc of radius $r_\omega = 0.25 \lambda$ with material properties $A = 3I$, $n = 1$, and located in the component of radii $r_2$ (see  Figure \ref{complexIncluded}-left).  The reconstruction using the indicator function $\I^{\delta}(z)$ is represented in Figure \ref{complexIncluded}-right. We can see in this example that we reconstruct periodic copies of the background component that contain the defect as predicted by the theory. We also observe numerically that the values of the indicator function are very different in the period that contain the defect. This means that, although not indicated by the theory, we numerically can determine the period that contains the defect. 

\begin{figure}[H] 
\centerline{ \begin{tabular}{cc}
\includegraphics[width=0.45\textwidth]{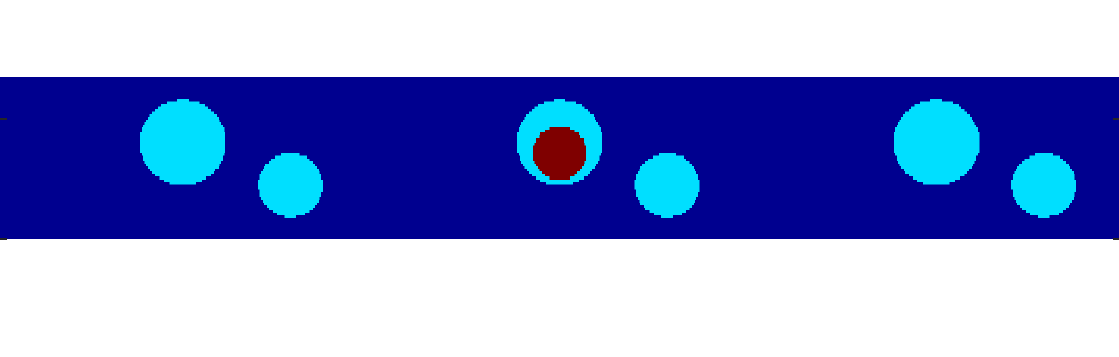}  & \includegraphics[width=0.45\textwidth]{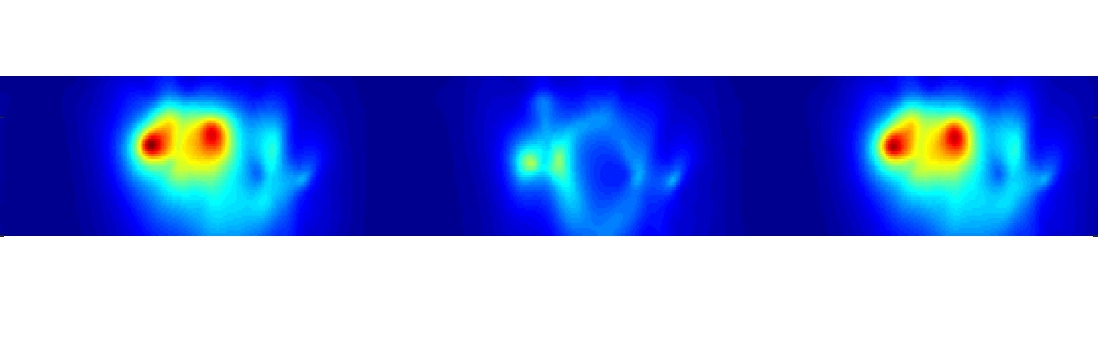} 
\end{tabular}}
\caption{Left: The exact geometry for Example 1. Right: The reconstruction using $z \mapsto
  \I^\delta(z)$}
\label{complexIncluded}
\end{figure}

\subsubsection*{Example 2.} In the second example, we consider the perturbation $\domp$ as in  Example 1 but now located such that $\domp$ has nonempty intersection with $D_p$ but not included in $D_p$ (see  Figure \ref{Composite2.2_structure}-left).  We consider  the refractive index of the defect which now is inhomogeneous. In particular, the refractive  index of the defect is  $A = 3I$ in $\domp \cap D_p$ and $A = 2I$ in $\domp \setminus D_p$. The  reconstruction is represented in Figure \ref{Composite2.2_structure}--right. We have the same conclusion and we additionally better see the part that lies outside the background  components.

\begin{figure}[H] 
\centerline{\begin{tabular}{cc}
\includegraphics[width=0.45\textwidth]{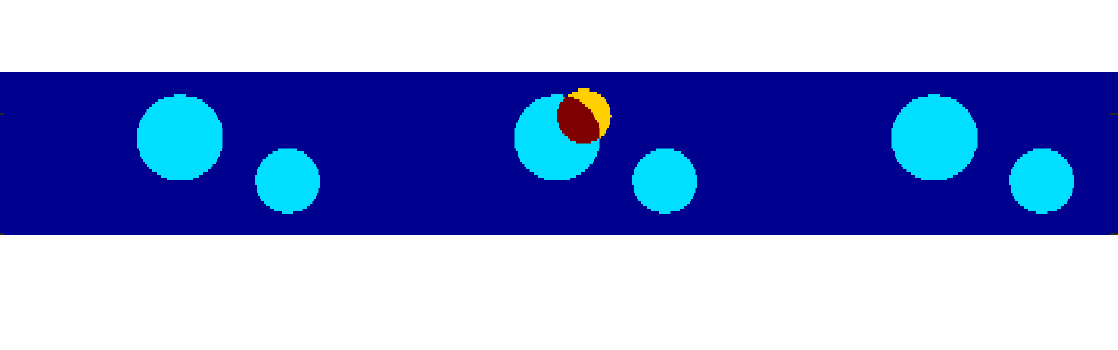} &
\includegraphics[width=0.45\textwidth]{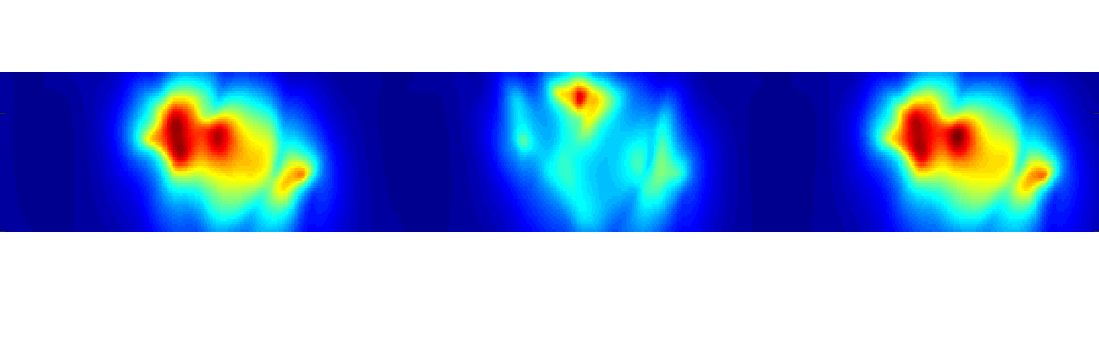} \\
\end{tabular}}
\caption{Left: The exact geometry for Example 2. Right: The reconstruction of the local perturbation using $z \mapsto
  \I^\delta(z)$}
%\label{Composite2_pertubation}
\label{Composite2.2_structure}
\end{figure}

\subsubsection*{Example 3.}
This example shows that when the defect has no intersection with the periodic background, the indicator function $\I^{\delta}(z)$ allows to reconstruct the true defect including its true location in the periodic medium. Here the defect is a disc of $r_\omega=0.25 \lambda$  with $A=2I$.  

\begin{figure}[H] 
% \hspace{-7mm}
\centerline{\begin{tabular}{cc}
\includegraphics[width=0.45\textwidth]{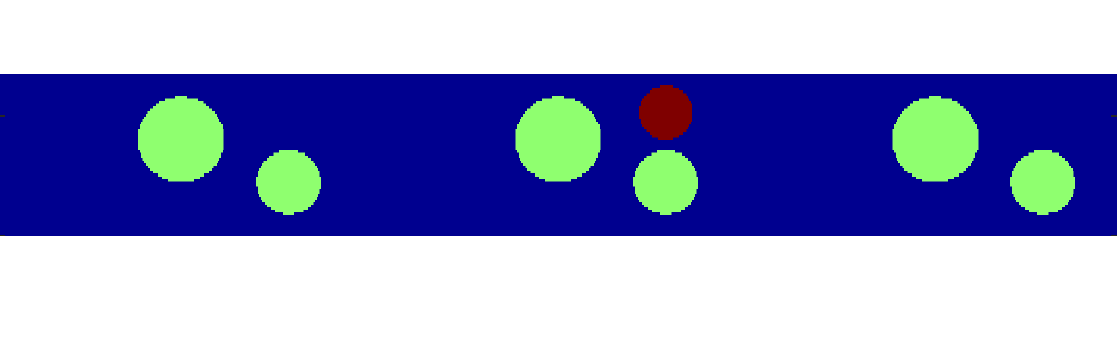} &
\includegraphics[width=0.45\textwidth]{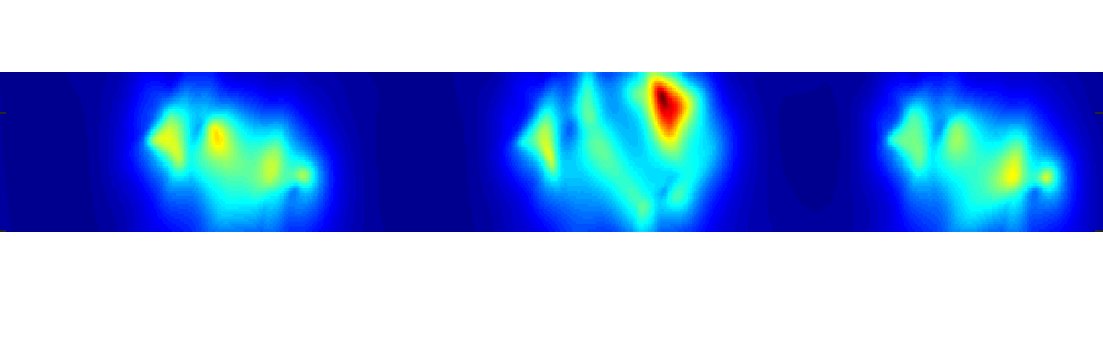}
\end{tabular}}
\caption{Left: The exact geometry for Example 3. Right: The reconstruction using $z \mapsto
  \I^\delta(z)$}
\label{Composite3_structure}
\end{figure}

\noindent
As a conclusion we observe that our numerical examples validate the theoretical prediction provided by Theorem \ref{DiffImagFunc} and produce similar reconstructions as in the case $A=1$ treated in \cite{Thi-Phong3, Thi-Phong4}. The case when the defect is entirely included in a component of the periodic background is theoretically ambiguous in the sense that the cell where the defect is embedded in cannot be determined accurately. However, we numerically observed that also in this case, one is able to detect the location of the period that contains the defect. 

\section*{Acknowledgements}
The research of T-P. Nguyen is supported in part by NSF Grant DMS-1813492.

 \section*{References}
 
% \bibliographystyle{siam}
%\bibliography{TPN-biblio} 

\begin{thebibliography}{10}
\bibitem{Arens2010}
{\sc T.~Arens}, {\em Scattering by biperiodic layered media: The integral
  equation approach}.
\newblock Habilitation Thesis, Universit{\"a}t Karlsruhe, 2010.

\bibitem{Arens2005}
{\sc T.~Arens and N.~Grinberg}, {\em A complete factorization method for
  scattering by periodic structures}, Computing, 75 (2005), pp.~111--132.

\bibitem{Audib2015a}
{\sc L.~Audibert}, {\em {Qualitative methods for heterogeneous media}}, theses,
  {Ecole Doctorale Polytechnique}, Sept. 2015.

\bibitem{Audib2015}
{\sc L.~Audibert, A.~Girard, and H.~Haddar}, {\em Identifying defects in an
  unknown background using differential measurements}, Inverse Problems and
  Imaging, 9 (2015), pp.~625--643.

\bibitem{Audib2014}
{\sc L.~Audibert and H.~Haddar}, {\em A generalized formulation of the linear
  sampling method with exact characterization of targets in terms of farfield
  measurements}, Inverse Problems, 30 (2014), p.~035011.

\bibitem{Bonne2011a}
{\sc A.-S. Bonnet-BenDhia, L.~Chesnel, and H.~Haddar}, {\em On the use of
  {T}-coercivity to study the interior transmission eigenvalue problem.}, C. R.
  Acad. Sci. Mathematics, 11-12 (2011), pp.~647--651.

\bibitem{Bourg2014}
{\sc L.~Bourgeois and S.~Fliss}, {\em On the identification of defects in a
  periodic waveguide from far field data}, Inverse Problems, 30 (2014),
  p.~095004.

\bibitem{CCH}
{\sc F.~Cakoni, D.~Colton, and H.~Haddar}, {\em Inverse scattering theory and
  transmission eigenvalues}, vol.~88 of CBMS-NSF Regional Conference Series in
  Applied Mathematics, Society for Industrial and Applied Mathematics (SIAM),
  Philadelphia, PA, 2016.

\bibitem{Thi-Phong4}
{\sc F.~Cakoni, H.~Haddar, and T.-P. Nguyen}, {\em {Analysis and Applications
  of Interior Transmission Problems Associated with Single Floquet-Bloch Mode
  Imaging (2018) (preprint, draft available upon request)}}.

\bibitem{Elsch2011a}
{\sc J.~Elschner and G.~Hu}, {\em Inverse scattering of elastic waves by
  periodic structures: uniqueness under the third or fourth kind boundary
  conditions}, Methods and Applications of Analysis, 18 (2011), pp.~215--244.

\bibitem{Thi-Phong2}
{\sc H.~Haddar and T.-P. Nguyen}, {\em {A volume integral method for solving
  scattering problems from locally perturbed infinite periodic layers}},
  {Applicable Analysis},  (2016), pp.~130 --158.

\bibitem{Thi-Phong3}
\leavevmode\vrule height 2pt depth -1.6pt width 23pt, {\em {Sampling methods
  for reconstructing the geometry of a local perturbation in unknown periodic
  layers}}, {Computers and Mathematics with Applications}, 74 (2017),
  pp.~2831--2855.

\bibitem{Lechl2013b}
{\sc A.~Lechleiter and D.-L. Nguyen}, {\em Factorization {Method} for
  {Electromagnetic} {Inverse} {Scattering} from {Biperiodic} {Structures}},
  SIAM Journal on Imaging Sciences, 6 (2013), pp.~1111--1139.

\bibitem{armin}
{\sc A.~Lechleiter and R.~Zhang}, {\em Reconstruction of local perturbations in
  periodic surfaces}, Inverse Problems, 34 (2018), pp.~035006, 17.

\bibitem{nguye2012}
{\sc D.~L. Nguyen}, {\em Spectral Methods for Direct and Inverse Scattering
  from Periodic Structures}, PhD thesis, {Ecole Polytechnique X}, 2012.

\bibitem{tpnguyen}
{\sc T.~P. Nguyen}, {\em {Direct and inverse solvers for scattering problems
  from locally perturbed infinite periodic layers}}, theses, {Universit{\'e}
  Paris-Saclay}, Jan. 2017.
\end{thebibliography}

\end{document}